\documentclass[12pt]{amsart}
\usepackage{amsmath}
\usepackage{amsxtra}
\usepackage{amscd}
\usepackage{amsthm}
\usepackage{amsfonts}
\usepackage{amssymb}
\usepackage{mathrsfs}
\usepackage[all]{xy}
\usepackage[pdftex]{color,graphicx}
\usepackage{graphicx}
\usepackage{color}
\usepackage{mathrsfs}

\newtheorem{thm}{Theorem}[section]
\newtheorem{lemma}[thm]{Lemma}
\newtheorem{prop}[thm]{Proposition}
\newtheorem{corol}[thm]{Corollary}

\newtheorem*{dfn}{Definition}

\newtheorem*{thmA}{Theorem A}
\newtheorem*{thmB}{Theorem B}
\newtheorem*{thmC}{Theorem C}

\def\Z{\Bbb Z}
\def\F{\Bbb F}
\def\Q{\Bbb Q}
\def\C{\Bbb C}
\def\N{\Bbb N}

\def\L{\mathcal L}

\DeclareMathOperator\aut{Aut} \DeclareMathOperator\inn{Inn}
\DeclareMathOperator\id{id} \DeclareMathOperator\cen{Z}
 
\DeclareMathOperator\Aut{Aut}

\DeclareMathOperator\Mod{mod} \DeclareMathOperator\prob{Prob_\Sigma}
\DeclareMathOperator\sym{Sym} \DeclareMathOperator\proba{Prob}
\DeclareMathOperator\GL{GL} 
 \DeclareMathOperator\Cay{Cay}
\DeclareMathOperator\ag{G} \DeclareMathOperator\SL{SL}
\DeclareMathOperator\V{Var} \DeclareMathOperator\E{E}
\DeclareMathOperator\ord{ord} 
\DeclareMathOperator\G{G} \DeclareMathOperator\walks{W}

\makeatletter

\newcommand{\Rmnum}[1]{\expandafter\@slowromancap\romannumeral #1@}
\makeatother

\begin{document}

\begin{abstract} A general sieve method for groups is formulated. It
enables one to ``measure" subsets of a finitely generated group. As
an application we show that if $\Gamma$ is a finitely generated non
virtually-solvable linear group of characteristic zero then the set
of proper powers in $\Gamma$ is exponentially small. This is a far
reaching strengthening of the main result of \cite{HKLS}.
\end{abstract}

\title{Sieve methods in group theory \Rmnum{1}: \\  Powers in Linear groups }


\author{Alexander Lubotzky and Chen Meiri}
\address{Einstein Institute of Mathematics \\Hebrew University\\Jerusalem 90914, Israel}
\email{alexlub@math.huji.ac.il, chen.meiri@mail.huji.ac.il}
\maketitle
\date{\today}

\section{introduction}

The sieve method is a classic one in number theory (see, for
example, \cite{FI}). Recently it found some applications in
non-commutative setting. On the one hand,  Bourgain-Gumburd-Sarnak
\cite{BGS1} applied it in studying almost-prime vectors in orbits of
non-commutative groups acting on $\Z^n$. On the other hand, Rivin
\cite{Ri} and Kowalski \cite{Ko} used it to study generic properties
of elements in the mapping class group and arithmetic groups. Our
formulation of the sieve method generalizes and simplifies the
second one and usually falls under the name `Large Sieve'. The goal
of this introduction is to state the general large sieve setting
with respect to group theory and to serve as a guideline for the
proof of Theorem A below. We start by describing the algebraic
problem and its background. After that, we explain how random walks
and sieve methods are used in its solution.

A \emph{virtually nilpotent group} is a group which contains a
nilpotent subgroup of finite index. Mal'cev proved:

\begin{thm}[Mal'cev \cite{Mal}]\label{mal} Let $\Gamma$ be a finitely generated virtually nilpotent
group. Then, for every $m \ge 1$ the set of $m$-powers
$\Gamma^m:=\{g^m \mid g \in \Gamma\}$ contains a finite index
subgroup of $\Gamma$.
\end{thm}

The converse is not true, even not for finitely generated groups:
For every prime $p$, Golod and Shafarevich built a finitely
generated residually finite  infinite group $\Gamma$ such that the
order of every element of $\Gamma$ is a power of $p$. In particular,
if $m$ is coprime to $p$ then any element of $\Gamma$ is an
$m$-power.

Yet, in \cite{HKLS} two kinds of partial converse results were
proved:

\begin{thm}[Hrushovski-Kropholler-Lubotzky-Shalev \cite{HKLS}]\label{Thm A}
Let $\Gamma$ be a virtually solvable group. If $n \ge 2$ and
$\cup_{m=2}^n \Gamma^{m}$ contains a finite index subgroup of
$\Gamma$ then $\Gamma$ is virtually nilpotent. On the other hand,
there exists a solvable group $\Gamma$ which is not virtually
nilpotent and still there exists $m \ge 2$ such that $\Gamma^m$
contains a coset of a finite index subgroup.
\end{thm}
\begin{thm}[Hrushovski-Kropholler-Lubotzky-Shalev \cite{HKLS}]\label{Thm B}
Let $\Gamma$ be a finitely generated linear group. If $n \ge 2$ and
finitely many tessellates of $\cup_{m=2}^n \Gamma^{m}$ cover
$\Gamma$ then $\Gamma$ is virtually solvable.
\end{thm}

The formulation of the second theorem of \cite{HKLS} suggests a
stronger result, i.e. does the fact that finitely many tessellates
of the set of {\bf all} proper powers $\cup_{m=2}^\infty \Gamma^{m}$
cover $\Gamma$ imply that $\Gamma$ is virtually solvable? But the
methods of \cite{HKLS}  are not suitable for handling all proper
powers together. The reason is that the proof uses only local data
of $\Gamma$, i.e., the images of $\cup_{m=2}^n \Gamma^{m}$ in finite
quotients of $\Gamma$, and it is clear that the image of
$\cup_{m=2}^\infty \Gamma^{m}$ in such a quotient is the full
quotient group (take $m$ to be coprime to the size of the finite
quotient). Thus, to extend the theorem to the case of all proper
powers one need to combine the local data with some global data on
the group and its elements. A way to do this is to use random walks
as we shell explain below.

A finite subset $\Sigma$ of $\Gamma$ is called \emph{admissible} if
it is symmetric, i.e $\Sigma=\Sigma^{-1}$, and the Cayley graph
$\Cay(\Gamma,\Sigma)$ is not bi-partite. Let $\Sigma$ be an
admissible generating subset of $\Gamma$. Random walks on
$\textrm{Cay}(\Gamma,\Sigma)$ can be used to `measure' a subset
$Z\subseteq \Gamma$ by estimating the probability $\prob(w_k \in Z)$
that the $k^{\text{th}}$-step of a random walk belongs to $Z$ for
larger and larger values of $k$'s. We say that $Z$ is
\emph{exponentially small with respect to $\Sigma$} if there exist
constants $c,\alpha>0$ such that $\prob(w_k\in Z) \le ce^{-\alpha
k}$ for all $k \in \N$. The set $Z$ is called \emph{exponentially
small} if it is exponentially small with respect to all admissible
generating subsets. It is not hard to see that if a subset is
exponentially small then finitely many tessellates of it cannot
cover $\Gamma$. Thus, our first theorem is the desired extension:

\begin{thmA}\label{power theorem} Let $\Gamma$ be a finitely generated subgroup of $\GL_n(\C)$ which
is not virtually-solvable. Then the set of proper powers
$\overset{\infty}{\underset{m=2}\cup} \Gamma^m$  is exponentially
small in $\Gamma$.
\end{thmA}

Despite of Theorem \ref{Thm B}, Theorem A is somewhat surprising:
The set $\cup_{m=2}^\infty \Gamma^m$ is dense in the profinite
topology of $\Gamma$ (as explained above) and still it is
exponentially small.

A reduction process, given in subsection \ref{reduction},  shows
that it is enough to prove the claim for a finitely generated
subgroup $\Gamma$ of $\GL_n(\Q)$ whose Zariski-closure is
semisimple. In order to avoid the technical difficulties we will
focus for now on the case where $\Gamma$ is a Zariski-dense subgroup
of $\SL_n(\Q)$. Fix such a subgroup $\Gamma$ and  some admissible
generating  subset $\Sigma$. For $k\in \N$ only the elements of
$\Gamma$ which belong to the ball of radius $k$,
$\mathcal{B}_\Sigma(k)$, can occur as the $k^{\text{th}}$-step of a
random walk on $\textrm{Cay}(\Gamma,\Sigma)$. As in \cite{LMR}, we
use in Lemma \ref{two parts} arguments involving matrices norms and
number theory to show that if $k$ is large enough and $g \in
\mathcal{B}_\Sigma(k)$ is a proper power then $g$ is either
virtually-unipotent (i.e. some positive power of $g$ is unipotent)
or $g$ is an $m$-power for some $2 \le m \le k^2$ (in fact, for some
$2 \le m \le ck$ where $c$ is a constant depending on $\Sigma$).

The advantage in considering random walks is now clear. By
considering only  proper powers which can occur at the
$k^{\text{th}}$-step of a random walk we only have to look at
virtually-unipotent elements and at finitely many powers. Thus, we
can divide the proof into two parts, each of them can be proven by
looking at the finite quotients.

The first part is to show that the set of all virtually unipotent
elements (not necessarily proper powers) is exponentially small.
Proposition \ref{subvarity} shows that if $V(\C)$ is a proper
subvariety of $\SL_n(\C)$ defined over $\Q$ then $V(\C)\cap\Gamma $
is exponentially small. In turn, Corollary \ref{unipotent} implies
that the set of virtually unipotent elements of $\Gamma$ is
contained in a proper subvariety of $\SL_n(\C)$ defined over $\Q$.

The second part is to find  positive constants $\gamma$ and $r$ such
that for every $k \ge r$ and every $2 \le m\le k^2$ we have
$\prob(w_k\in\Gamma^m) \le e^{-\gamma k}$. Indeed, if such constants
exist and $\alpha$ is a positive constant strictly smaller the
$\gamma$ then for large enough $k$ $$\prob(w_k \in\cup_{2 \le m \le
k^2}\Gamma^m) \le k^2e^{-\gamma k} \le e^{-\alpha k }.$$ However,
for every $m\ge 1$ the set of $m$-powers is Zariski-dense in
$\SL_n(\C)$ so we can not use the same argument as for the first
part. This brings us to the large sieve method. Corollary \ref{group
large sieve} gives the general quantitative statement of this method
while Theorem B stated below is an easy consequence which is
suitable for now,

\begin{thmB} Let $\Gamma$ be a group generated by a finite symmetric set $\Sigma$ for which
$\Cay(\Gamma,\Sigma)$ is not bi-partite. Let $(N_j)_{j\ge 2}$ be a
family of normal finite index subgroups. Let $Z \subseteq \Gamma$
and assume that:
\begin{itemize}
\item[1.]  $\Gamma$ has property-$\tau$ w.r.t to the family $\{N_i \cap N_j \mid i,j \ge
2\}$.
\item[2.] The sequence $(|\Gamma/N_j|)_{j \ge 2}$ grows polynomially in
$j$.
\item[3.] $|\Gamma/N_i \cap N_j|=|\Gamma/N_i||\Gamma/N_j|$ for distinct $i$ and
$j$.
\item[4.] There exists $c>0$ such that $|ZN_j/N_j|\le (1-c)|\Gamma/N_j|$ for
every $j \ge 2$.
\end{itemize}
Then there are positive constants $\gamma$ and $r$ depending only on
$\Sigma$, $c$ and the growth of $(|\Gamma/N_j|)_{j \ge 2}$ such that
$\prob(w_k \in Z)\le e^{-\gamma k}$ for every $k \ge r$. In
particular, $Z$ is exponentially small.
\end{thmB}

Returning to our case, $\Gamma$ is a Zariski-dense subgroup of
$\SL_n(Z)$. Fix $m\ge 2$ and define $Z_m$ to be the set of
$m$-powers. We start by verifying the conditions of Theorem B in
order to bound $\prob(w_k\in Z_m)$. For every prime $p$ define
$M_p:=\ker \pi_p$ where $\pi_p:\Gamma \rightarrow \SL_n(\Z/p\Z)$ is
the modulo-$p$ homomorphism. Let $\mathcal{P}$ be the set of primes.
Then, $\Gamma$ has property-$\tau$ w.r.t to the family $\{M_p \cap
M_q \mid p,q \in \mathcal{P}\}$ by the recent result of
Salehi-Golsefidy and Varju \cite{SGV}. In fact, for our special case
the work Varju \cite{Va} is enough. The Strong Approximation Theorem
of Wiesfeiler \cite{We} and Nori \cite{No} implies that if $p$ and
$q$ are distinct large enough primes then
$\pi_{pq}(\Gamma)=\SL_n(\Z/pq\Z)$. Recall that $\SL_n(\Z/pq\Z)$ is
isomorphic to $\SL_n(\Z/p\Z)\times\SL_n(\Z/p\Z)$ so $|\Gamma/M_p
\cap M_q|=|\Gamma/M_p||\Gamma/M_q|$ for distinct large enough primes
$p$ and $q$. Lemma \ref{powers lemma} shows that if $p$ is a prime
which belongs to the sequence $(1+mi)_{i \ge s}$ where
$s:=3\binom{n}{2}$ then $|\pi_p(Z_m)| \le
\left(1-\frac{1}{6n!}\right)|\pi_p(\Gamma)|$, so we can take
$c:=\frac{1}{6n!}$. Let $(p_i)_{i \ge 2}$ be an ascending
enumeration of the large enough primes which belong to the sequence
$(1+mi)_{i \ge s}$ and define  $N_i:=\ker \pi_{p_i}$. Then, $\Gamma$
has property-$\tau$ w.r.t to the family $\{N_i \cap N_j \mid i,j \ge
2\}$ and $|ZN_j/N_j|\le (1-c)|\Gamma/N_j|$ for every $j \ge 2$. The
last condition left to verify is that $(|\Gamma/N_j|)_{j \ge 2}$
grows polynomially. Since $|\Gamma/N_j|\le p_j^{n^2}$ the question
about the growth of $(|\Gamma/N_j|)_{j \ge 2}$ translates to a
question about the density of the primes in the sequence $(1+mi)_{i
\ge s}$. We use a quantitative version of Dirichlet's Theorem about
primes in arithmetic progression (see Theorem \ref{qdt} and
Corollary \ref{qdl} below) to estimate this density and to show that
$(|\Gamma/N_j|)_{j \ge 2}$ grows polynomially. Thus all the
conditions of Theorem B hold and there are positive constants
$\gamma$ and $r$ for which $\prob(w_k \in Z_m)\le e^{-\gamma k}$ for
every $k \ge r$. In fact, a more detailed study of the growth of
$(|\Gamma/N_i|)_{i \ge 2}$ shows that the $\gamma$ and $r$ are
independent of $m$ provided $m \le k^2$. The (sketch of the) proof
of the special case is now complete.

The main difficulty in the proof of the general case lies in proving
the existence of the constant $c$ needed in Theorem B. This requires
a delicate analysis of the automorphism groups of almost simple
groups. We dedicate Section \ref{chap liegroups} to this analysis,
which though quite technical, might be of interest on its on. A
non-quantitative version of main theorem of Section \ref{chap
liegroups} (Theorem \ref{power main4}) is:

\begin{thmC}\label{power main1} Fix $d,l,r \in \N^+$. Then there are coprime $a,b \in \N^+$
and a positive constant $c$ such that for every prime $p$ which
belongs the arithmetic sequence  $(a+bj)_{j \ge 1}$ the following
claim holds:

Let $G$ be a finite group with a non-trivial normal subgroup $H$.
Furthermore, assume that $H$ is isomorphic to the direct product of
at most $r$ copies of finite simple groups of Lie type of rank $l$
over the field $\F_{p^d}$ with ${p^d}$ elements. Then for every
coset $L$ of $H$ we have $|\{g^m \mid g \in L\}| \le (1-c)|L|$.
\end{thmC}

The current paper is a first in a series of three (see \cite{LuMe}
and \cite{LuMe2}) in which the same sieve method is applied to
obtain results on the mapping class group and on the automorphism
group of a free group, respectiviliy. It seems that Theorem B has
the potential of having more applications in group theory (see also
\cite{Lu2}).

{\bf Acknowledgments.} The authors are grateful to the ERC and the
ISF for partial support. They also want to thank Emmanuel
Breuillard, Dorian Goldfeld, Emmanuel Kowalski and Ron Peled for
useful conversations.

\section{Random walks, expanders and unipotent elements}\label{chap tau}

\subsection{Random Walks}

\begin{dfn} Let $\Gamma$ be a group. A multi-subset $\Sigma$ of $\Gamma$ is
called \emph{symmetric} if for every $s \in \Sigma$ the number of
times $s$ occurs in $\Sigma$ equals to the number of times $s^{-1}$
occurs in $\Sigma$. A finite symmetric multi-subset $\Sigma$ of
$\Gamma$ is called \emph{admissible} if the Cayley graph
$\Cay(\Gamma,\Sigma)$ is not bi-partite, e.g. the identity belongs
to $\Sigma$.
\end{dfn}

Fix an admissible generating multi-subset
$\Sigma=[s_1,\ldots,s_{|\Sigma|}]$ of a group $\Gamma$. Since
$\Sigma$ is a multi-set, $\Cay(\Gamma,\Sigma)$ might contains self
loops and multiple edges. Let $\bar{\Sigma}\subseteq \Gamma$
consists of the elements which belongs to $\Sigma$. Note that
$\Cay(\Gamma,\Sigma)$ is connected if and only if $\bar{\Sigma}$
generates $\Gamma$ and $\Cay(\Gamma,\Sigma)$ is bi-partite if and
only if $\bar{\Sigma}$ satisfies an odd relation, e.g. contains the
identity.

A \emph{walk on $\Cay(\Gamma,\Sigma)$} is an infinite sequence of
edges $(e_k)_{k \in \N^+}$ such that the initial vertex of $e_1$ is
the identity and the terminal vertex of $e_k$ is the initial vertex
of $e_{k+1}$ for every $k \in \N^+$. The initial vertex of $e_{k+1}$
is called the $k^{\textrm{th}}$-step of the walk, in particular,
$w_0$ is the identity. The set of walks $\walks(\Sigma)$ can be
identified with $\{1,\ldots,|\Sigma|\}^{\N}$. Hence, the uniform
probability measure of $\{1,\ldots,|\Sigma|\}$ induces a structure
of a probability space on $\walks(\Sigma)$. Once $\walks(\Sigma)$
becomes a probability space, the term `random walk' has a natural
meaning. However, it is sometimes useful to think of a random walk
on $\Cay(\Gamma,\Sigma)$ as constructed as follows: The random walk
starts at the identity and if it arrives to a vertex $v$ at some
step then in the next step it will use an edge which starts at $v$
and every such an edge as probability $\frac{1}{|\Sigma|}$ to be
used.

For a subset $Z$ of $\Gamma$ we denote the probability that the
$k^{\textrm{th}}$-step of a walk belongs to $Z$ by $\prob(w_k \in
Z)$. Of course, for a general $Z$ the limit $\lim_{k \rightarrow
\infty} \prob(w_k \in Z)$ might not exist and even if it does, this
limit might depend on $\Sigma$. However, in many natural cases this
limit exists and does not depend on $\Sigma$. For example if $Z$ is
a finite index subgroup of $\Gamma$ then always $\lim_{k \rightarrow
\infty} \prob(w_k \in Z)=[\Gamma:Z]^{-1}$.

In the sequel we will present and apply the large sieve method which
helps to show that certain subsets are `very small'. More formally,
a subset $Z\subseteq \Gamma$ is called \emph{exponentially small
with respect to $\Sigma$} if there are positive constants $c$ and
$\alpha$ such that $\prob(w_k \in Z) \le ce^{-\alpha k}$ for every
$k \in \N$. This means that not only that this limit is zero but
also that there is an `exponentially fast' convergence to this
limit. A set is \emph{exponentially small} if it exponentially small
with respect to every admissible generating multi-subset of
$\Gamma$.

The reason that we chose to work with a multi-sets $\Sigma $ instead
of the underlying set $\bar{\Sigma}$ is that when passing from the
group $\Gamma$ to a quotient, where the image of $Z$ is
exponentially small, we want to deduce that $Z$ ia exponentially
small. However, the quotient homomorphism does not have to be
injective on the generating set so we regard its image as a
multi-set.

\subsection{Property-$\tau$}

Let us now define expanders and property-$\tau$. A more detailed
discussion about these subject can be found in \cite{HLW} and
\cite{Lu}. Let $X$ be an undirected $d$-regular graph, self loops
and multiple edges are allowed. Let $n$ be the number of vertices of
$X$. The normalized adjacency matrix of $X$, denoted by $A_X$, is an
$n \times n$ matrix whose $(v,u)$ entry is $\frac{d_{v,u}}{d}$ where
$d_{v,u}$ is the number of edges in $X$ between vertex $u$ and
vertex $v$. Being real and symmetric, the matrix $A_X$ has $n$ real
eigenvalues which we denote by $\lambda_1 \ge \lambda_2 \ge \cdots
\ge \lambda_n$. It is not difficult to see that all the eigenvalues
of $A_X$ lie between $-1$ and $1$. More precisely, $\lambda_1=1$ and
if the graph is connected then $\lambda_2<1$. The \emph{spectral
gap} of $X$ is defined to be $1-\lambda_2$.  A family of graphs is
called \emph{$\varepsilon$-expander} if the spectral gap of every
graph in this family is at least $\varepsilon$. It is called an
\emph{expander} if it is $\varepsilon$-expander for some
$\varepsilon>0$

The following lemma is a group theoretic formulation of Theorem 3.2
of \cite{HLW}.
\begin{lemma}\label{expander lemma} Let $\Gamma$ be a finite group with a finite symmetric generating multi-set $\Sigma$.
Let $A$ be the normalized adjacency metrix of of the Cayley graph
$\textrm{Cay}(\Gamma,\Sigma)$ and denote the eigenvalues of $A$ by
$\lambda_1 \ge \lambda_2 \ge \cdots \ge \lambda_n$. Define
$\alpha:=\max(|\lambda_2|,|\lambda_n|)$. Then for every subset $T
\subseteq \Gamma$ and every $k \in \N^+$ we have:
$$\left| \prob(w_k \in T)-\frac{|T|}{|\Gamma|}\right| \le
\sqrt{|\Gamma|}\alpha^k. $$ \begin{flushright}
$\Box$\end{flushright}
\end{lemma}
Let $\Gamma$ be a finitely generated group and let $\mathcal{N}$ be
a family of finite index normal subgroups of $\Gamma$. The group
$\Gamma$ has \emph{property-$\tau$} w.r.t $\mathcal{N}$ if for some
finite symmetric generating multi-set $\Sigma$ and some
$\varepsilon>0$ the family of Cayley graphs
$\{\mathrm{Cay}(\Gamma_N,\Sigma_N)\mid N \in \mathcal{N}\}$ is an
$\varepsilon$-expander where $\Gamma_N$ and $\Sigma_N$ are the
images of $\Gamma$ and $\Sigma$ under the quotient homomorphism
$\Gamma \rightarrow \Gamma/N$ ($\Sigma_N$ is a multi-set). It is not
difficult to show that if $\Gamma$ has property-$\tau$ w.r.t
$\mathcal{N}$ then for every finite symmetric generating multi-set
$\Sigma$ there is $\varepsilon>0$ such that the family of Cayley
graphs $\{\mathrm{Cay}(\Gamma_N,\Sigma_N)\mid N \in \mathcal{N}\}$
is an $\varepsilon$-expander, however, $\varepsilon$ may depend on
the generating multi-set. The maximal $\varepsilon$ for which the
family $\{\mathrm{Cay}(\Gamma_N,\Sigma_N)\mid N \in \mathcal{N}\}$
is an $\varepsilon$-expander is called the \emph{expansion constant
of $\Sigma$}.

We want to apply Lemma \ref{expander lemma} to groups with
property-$\tau$. However, we have to be a little bit cautious. The
expander property bounds  the second largest eigenvalue of the
normalized adjacency of the above graphs, but it does not tell us
anything about the absolute value of the smallest eigenvalue. Still,
in our case this can be overcome.

\begin{lemma}\label{eigenvalue} Let $\Sigma$ be a symmetric generating multi-set of
$\Gamma$. If $\Cay(\Gamma,\Sigma)$ is not bi-partite then there is
some constant $c>-1$ such that for every finite index normal
subgroup $N$ of $\Gamma$ the smallest eigenvalue of the normalized
adjacency matrix of $\mathrm{Cay}(\Gamma_N,\Sigma_N)$ is greater
than $c$. In fact, $c$ depends only on the size of $\Sigma$ and the
length of the shortest odd cycle in $\Cay(\Gamma,\Sigma)$.
\end{lemma}
\begin{proof} Let $N$ be a finite index normal subgroup of $\Gamma$.
Let $X$ be the Cayley graph of $\Gamma_N$ with respect to
$\Sigma_N$. Let $l \in \N$ be the length of the shortest odd cycle
in $\Cay(\Gamma,\Sigma)$. The entries on the the diagonal of $A^l$
are positive and equal to each other. In fact, the value of the
entries is at least $\frac{1}{|\Sigma|^l}$ since it is equal to the
probability that a random walk on $X$ starting at some vertex return
to this vertex at the $l^{\text{th}}$-step. Hence, we can write
$A^l$ as $\alpha I+ B$ where $\alpha \ge\frac{1}{|\Sigma|^l}$, $I$
is the identity matrix and $B$ a real symmetric matrix with
non-negative entries such that the sum of the entries in every row
and column is $1-\alpha$. Thus, the smallest eigenvalue of $B$ is at
least $\alpha-1$ which implies that the smallest eigenvalue of $A^l$
is at least $2\alpha-1$. In turn, the smallest eigenvalue of $A$ is
at least $(2\alpha-1)^{\frac{1}{l}}\ge
(\frac{2}{|\Sigma|^l}-1)^{\frac{1}{l}}$ which is a real number
greater than $-1$ since $l$ is odd.
\end{proof}

It is well known that a connected $k$-regular graph is bi-partite if
and only if the smallest eigenvalue of the normalized adjacency
matrix of this graph is $-1$. Thus, Lemma \ref{eigenvalue} shows
that the condition that $\Cay(\Gamma,\Sigma)$ is not bi-partite does
not only imply that the graphs in $\{\mathrm{Cay}(\Gamma_N,\Sigma_N)
\mid N \in \mathcal{N}\}$ are not bi-partite but also that they are
`uniformly far' from being bi-partite.

A straightforward corollary of Lemmas \ref{expander lemma} and
\ref{eigenvalue} is:

\begin{corol}\label{property tau corol} Let $\Gamma$ be a finitely generated group
with an admissible generating multi-set  $\Sigma$. Let $\mathcal{N}$
be a family of finite index normal subgroups of $\Gamma$. Assume
that the family of Cayley graphs $\{\textrm{Cay}(\Gamma_N,\Sigma_N)
\mid N \in \mathcal{N} \}$ is an $\varepsilon$-expander for some
$\varepsilon>0$. For every $N \in \mathcal{N} $, let $\pi_N:\Gamma
\rightarrow \Gamma_N$ be the quotient homomorphism. Then there
exists $\delta>0$, depending only on $\varepsilon$, on the size of
$\Sigma$ and on the length of the shortest odd cycle in
$\Cay(\Gamma,\Sigma)$, such that for every $k\in \mathbb{N}$, every
$N \in \mathcal{N}$ and every $T \subseteq \Gamma_N$ the following
holds:
$$\left|\prob(\pi_N(w_k)\in T)-\frac{|{T}|}{|\Gamma_N|}\right|\le
\sqrt{|\Gamma_N|}e^{-\delta k}.$$
\end{corol}

\subsection{Linear groups}\label{lin gr}

Most of the upcoming applications of property-$\tau$ will be related
to linear groups. We will mainly use the following result of
Salehi-Golsefidy and Varju  which in turn is built on the work Varju
\cite{Va} and Bourgain-Gumbord-Sarnak \cite{BGS1} and on the Product
Theorem of Breuillard-Green-Tao \cite{BGT1} and of Pyber-Szab\'{o}
\cite{PS} who followed and generalized Helfgott \cite{He}.

\begin{thm}[Salehi-Golsefidy-Varju \cite{SGV}]\label{SG-V} Fix $q_0 \in \N^+$ and let $\Gamma \subseteq
\GL_n(\Z[\frac{1}{q_0}])$ be a finitely generated subgroup such that
the connected component of its Zariski-closure in $\GL_n(\C)$ is
prefect. There exists $q_1\in \N^+$ divisible by $q_0$ such that
$\Gamma$ has property-$\tau$ with respect to the family
$$\{N_d \mid d \text{ is a square free positive integer
coprime to } q_1  \}$$ where $N_d$ is the kernel of the homomorphism
$\pi_d:\Gamma \rightarrow GL_n(\Z/d\Z)$ induced by the residue map
$\Z[\frac{1}{q_0}] \rightarrow \Z/d\Z$.
\end{thm}

In practice, it is very helpful to know what is the image
$\pi_d(\Gamma)$. Define $$I_\Gamma:=\{f\in \Z[t_{i,j}]_{1 \le i.j
\le n} \mid \forall g \in \Gamma.\ f(g)=0\}.$$ The Zariski-closure
$\G(\C)$ of $\Gamma$ in $\GL_n(\C)$ equals $\{g \in \GL_n(\C)\mid
\forall f\in I_\Gamma.\ f(g)=0\}$. Let $d \in \N^+$ and define
$$\G(\Z/d\Z):=\{g \in \GL_n(\Z/d\Z) \mid \forall f \in I_\Gamma.\
\bar{f}(g)=0\}$$ where $\bar{f}$ is the image of $f$ under the
modulo-$d$ homomorphism. Clearly, $\pi_d(\Gamma)$is contained in
$\G(\Z/d\Z)$. The following Strong Approximation Theorem of
Wiesfeiler and Nori gives us sufficient conditions for equality to
hold.

\begin{thm}[Wiesfeiler \cite{We}, Nori \cite{No}]\label{Strong} Let $q_0\in \N^+$. Let $\Gamma$ be a subgroup of $\GL_n(\Z[\frac{1}{q_0}])$
such that the  Zariski-closure of it $\ag(\C)$ is semisimple,
connected and simply connected.  There is a number $q_2 \in \N^+$
divisible by $q_0$ such that $\pi_d(\Gamma)=\ag(\Z/d\Z)$ for every
$d \in \N^+$ coprime to $q_2$.
\end{thm}

A finitely generated subgroup of $\GL_n(\Q)$ is contained in
$\GL_n(\Z[\frac{1}{q_0}])$ for some $q_0\in \N^+$. Thus, Corollary
\ref{property tau corol} together with Theorems \ref{SG-V} and
\ref{Strong} imply:

\begin{corol}\label{tau corol} Let $\Gamma$ be a finitely generated Zariski-dense subgroup of
$SL_n(\Q)$ and let $\Sigma$ be an admissible generating
multi-subset. There exist $q_3 \in \N^+$ and $\delta>0$ such that
for every square free $d \le e^{\delta k}$ coprime to $q_3$ and
every $T \subseteq \SL_n(\Z/dZ)$, the probability
$\prob(\pi_d(w_k)\in T)$ is approximately
$\frac{|T|}{|\SL_n(\Z/d\Z)|}$ and the error term is smaller than
$e^{-\delta k}$ (so it decays exponentially fast with $k$).
\end{corol}

The following proposition, which gives an example of a family
consisting of exponentially small subsets, can be easily deduced
from Proposition 3.2 of \cite{BGS1}.

\begin{prop}[Bourgain-Gamburd-Sarnak,  \cite{BGS1}]\label{subvarity}  Let $\Gamma$ be a finitely generated Zariski-dense
subgroup of $\SL_n(\Q)$. Let $V$ be a proper subvariety of
$\SL_n(\C)$ defined over $\Q$. Then, $V(\C)\cap \Gamma$ is
exponentially small.
\end{prop}
\begin{proof} Fix an admissible generating  multi-subset $\Sigma$ and let $\delta$ be
as in Corollary \ref{tau corol}. Assume that $k$ is large enough and
pick a prime between $\frac{1}{2}e^{\delta k}$ and $e^{\delta k}$
coprime to $q_3$ for which $V(\Z/pZ)$ is defined. The dimension of
$V$ is at most $n^2-2$ so by the Lang-Weil estimates \cite{LW},
$|V(\Z/p\Z)| \le (c+1) p^{n^2-2}$ where $c$ is the number of
irreducible components of maximal dimension of $V$. Now,
$|\SL_n(\Z/p\Z)| \ge \frac{1}{2}p^{n^2}$ so
$\frac{|V(\Z/p\Z)|}{|\SL_n(\Z/q\Z)|} \le \frac{2c+2}{p}$. Corollary
\ref{tau corol} shows that the probability $\prob(\pi_p(w_k)\in
V(\Z/p\Z))$ is at most $(4c+5)e^{-\delta k}$.
\end{proof}
\begin{dfn}
An element $g \in \SL_n(\Z)$ is called unipotent if all is
eigenvalues are equal to 1. An element $g \in \SL_n(\Z)$ is called
virtually unipotent if for some $m \ge 1$ the element $g^m$ is
unipotent.
\end{dfn}

A straightforward Corollary of Proposition \ref{subvarity} is:

\begin{corol}\label{unipotent}  Let $\Gamma$ be a Zariski-dense
subgroup of $\SL_n(\Q)$. Then, the set of virtually unipotent
elements is exponentially small.
\end{corol}
\begin{proof} An element $g \in \SL_n(\Q)$ is virtually unipotent if
and only if all its eigenvalues are roots of unity. There are only
finitely many roots of unity which are roots of a monic polynomial
of degree at most $n$ over $\Q$. Hence, there is a constant $m \ge
2$, depending only on $n$, such that if $g \in \SL_n(\Q)$ is
virtually unipotent then $g^m$ is unipotent. The set of elements of
$\SL_n(\C)$ whose $m^{\text{th}}$-power is unipotent is a proper
subvariety defined over $\Z$. Proposition \ref{subvarity} completes
the proof.
\end{proof}

\section{The Large Sieve} \label{chap sieve}

In the previous section we proved that the set $U$ consisting of
virtually-unipotent elements of $\Gamma$ is exponentially small
where $\Gamma$ is a Zariski-dense subgroup of $\SL_n(\Z)$. For every
large enough $k\in \N^+$ we picked a large prime $p_k$ with two
properties. The first is that the set $\pi_{p_k}(U)$ is
exponentially small in $\pi_{p_k}(\Gamma)$, i.e., the ratio between
the size of $\pi_{p_k}(\Gamma)$ and the size of $\pi_{p_k}(U)$ grows
exponentially with $k$. The second is that the image of the
$k^{\text{th}}$-step of a random walk is `almost uniformly
distributed' in $\pi_p(\Gamma)$.

There are some exponentially small sets for which an argument of
that kind does not work. For example, it does not work for the set
of proper powers or even for the set of $m$-powers for some fixed
$m$. The reason is that there is a fixed positive proportion
$\alpha>0$ such that for every prime $p$ the proportion of the set
of $m$-powers in $\pi_p(\Gamma)$ is at least $\alpha$. In order to
overcome this problem one may look at the image of $\Gamma$ under
the modulo-$d$ homomorphism where $d$ is a product of a linear
number (as a function of $k$) of primes. This rises a new problem,
the image of the $k^{\text{th}}$-step of a random walk does not
`almost uniformly distributed' in $\pi_d(\Gamma)$.

The large sieve method provides a way to deal with this situation.
It implies that in order to show that a set $Z \subseteq \Gamma$ is
exponentially small it is enough to find a constant $c>0$ and
exponentially many (as a function of $k$) primes $p$ for which the
following three conditions hold. The first is that
$\frac{|\pi_p(Z)|}{|\pi_p(\Gamma)|} \le 1-c$ for every such a prime
$p$. The second is that the $k^{\text{th}}$-step of a random walk is
`almost uniformly distributed' in $\pi_p(\Gamma)$ for every such a
prime $p$. The third is that for every two distinct such primes $p$
and $q$ the images in $\pi_p(\Gamma)$ and $\pi_q(\Gamma)$ of the
$k^{\text{th}}$-step of a random walk are `almost independent'.
\\

\subsection{The Large Sieve Theorem}

We start this section by stating a lemma which has nothing to do
with property-$\tau$ nor with groups. Property-$\tau$ will come into
the picture once we try to use this proposition in the context of
group theory. We are thankful to Ron Peled who simplified the proof
of the next lemma by suggesting the use of Chebyshev's inequality.

\begin{lemma}\label{prob} Let $U$ be a probability space. Let
$(A_i)_{1 \le i \le L}$ be a series of events. For $1 \le i,j \le L$
denote:
$$W(i,j):=
{\proba(A_i\cap A_j)-\proba(A_i)\proba(A_j)},$$
$$\triangle:=\max_{1 \le i \ne j \le L} |W(i,j)|,$$
and
$$M:=\sum_{i=1}^L{\proba(A_i)}.$$
Then:
$$\proba(U\setminus\bigcup_{1\le i \le L}A_i)\le\frac{L+L^2\triangle }{M^2}.$$
\end{lemma}
\begin{proof} Chebyshev's inequality says that if $X$ is a random variable then for $C \ge
0$:
$$\proba(|X-\E(X)|)\ge C) \le \frac{\V(X)}{C^2}.$$
In particular, if $C=\E(X)$, then:
$$\proba(X=0)\le \proba(|X-\E(X)|\ge \E(X)) \le
\frac{\V(X)}{\E(X)^2}.$$ Define $X_i$ to be the indicator function
of $A_i$ and set $X:=\sum_{1 \le i \le L}X_i$. Then, $M=\E(X)$ while
$$\V(X):=\E((X-\E(X))^2)=\sum_{1 \le i , j \le
L}E(X_iX_j-2X_i\E(X_j)+E(X_i)E(X_j))=$$
$$\sum_{1 \le i , j \le
L}(E(X_iX_j)-E(X_i)E(X_j))= \sum_{1 \le i , j \le L} W(i,j) \le
L+L^2\triangle .$$ Thus,
$$\proba(U\setminus\bigcup_{0\le i \le L}A_i)=\proba(X=0) \le\frac{L+L^2\triangle}{M^2}.$$
\end{proof}

\begin{thm}\label{sieve theorem} Fix $s \ge 2$. Let $\Gamma$ be a finitely generated group and let
$\Sigma\subset \Gamma$ be an admissible multi-set. Let $(N_i)_{i \ge
2}$ be a series of finite index normal subgroups of $\Gamma$. For
$i,j \ge 2$ denote $N_{i,j}:=N_i\cap N_j$,
$\Gamma_{i,j}:=\Gamma/N_{i,j}$ and let $\pi_{i,j}:\Gamma\rightarrow
\Gamma_{i,j}$ be the quotient homomorphism. Let $Z\subseteq \Gamma$
and assume that there are positive constants $\delta$, $c$ and $d$
and a sequence $(T_i)_{i \ge 2}$ such that for every $i,j \ge s$ the
following conditions hold:
\begin{itemize}
\item[0.]  $T_i \subseteq \Gamma_{i,i} \setminus
\pi_{i,i}(Z)$.
\item[1.] For every
$T \subseteq \Gamma_{i,j}$, $\left|\prob(\pi_{i,j}(w_k)\in
T)-\frac{|{T}|}{|\Gamma_{i,j}|}\right|\le
\sqrt{|\Gamma_{i,j}|}e^{-\delta k} .$
\item[2.]  $|\Gamma_{i,i}| \le i^d$ .
\item[3.] If $i$ and $j$ are distinct then
$\frac{|\{gN_{i,j} \mid gN_i \in T_i \ \wedge \ gN_j \in T_j
\}|}{|\Gamma_{i,j}|} =
\frac{|T_i||T_j|}{|\Gamma_{i,i}||\Gamma_{j,j}|}$.
\item[4.] $\frac{|T_i|}{|\Gamma_{i,i}|} \ge c$.
\end{itemize}
Then for every $k \ge \frac{d+1}{\delta}\log(2s)$, $\prob(w_k \in Z)
\le \frac{20}{c^2}e^{-\frac{\delta k}{d+1}}.$ In particular, $Z$ is
exponentially small w.r.t $\Sigma$.
\end{thm}

\begin{proof} For every $i \in I$ denote $\Gamma_i:=\Gamma_{i,i}$ and $\pi_i:=\pi_{i,i}$.
Conditions 1 and 3 imply that for every distinct $i,j\in I$ and
$k\in\mathbb{N}^+$:
\begin{equation}\label{eq31}{} \left|\prob(\pi_i(w_k) \in T_i)-\frac{|T_i|}{|\Gamma_i|}\right|\le
\sqrt{|\Gamma_i|}e^{-\delta k}
\end{equation} and
\begin{equation}\label{eq32} \left|\prob(\pi_{i}(w_k)\in T_i \ \wedge \
\pi_{j}(w_k)\in
T_j)-\frac{|T_i|}{|\Gamma_i|}\frac{|T_j|}{|\Gamma_j|}\right|\le
\sqrt{|\Gamma_i||\Gamma_j|}e^{-\delta k}.
\end{equation}

Recall that the set of walks $\walks(\Sigma)$  is a probability
space. For $i\ge s$ and $k \ge 1$, denote $A_{i,k}:=\{w\in
\walks(\Sigma)\mid\pi_i(w_k)\in T_i \}$. Note that if $w$ is a walk
and $w_k\in Z$, then $w\not \in \cup_{i \ge s} A_{i,k}$. We can
rewrite equations (\ref{eq31}) and (\ref{eq32}) in the form:
\begin{equation}\label{eq51}\left|{\proba}(A_{i,k})-\frac{|T_i|}{|\Gamma_i|}\right|\le
\sqrt{|\Gamma_i|}e^{-\delta k}\end{equation} and
\begin{equation}\label{eq52}\left|{\proba}(A_{i,k}\cap A_{j,k})-\frac{|T_i|}{|\Gamma_i|}\frac{|T_j|}{|\Gamma_j|}\right|\le
\sqrt{|\Gamma_i||\Gamma_j|}e^{-\delta k}.\end{equation} Thus,
Equations \ref{eq51} and \ref{eq52} together with Condition 2 imply
that for $i \ne j$:
\begin{equation}\label{eq53}\left| {\proba}(A_{i,k}\cap A_{j,k}) -{\proba}(A_{i,k}){\proba}(A_{j,k})\right|\le
4{\max (i,j)}^de^{-\delta k}.
\end{equation}

Define $L_k:=e^{\frac{\delta k}{d+1}}$. Equation $(\ref{eq53})$
shows that for every distinct $s \le i,j \le L_k$ we have
$|W_k(i,j)|\le 4e^{-\frac{\delta k }{d+1}}$ where $W_k(i,j)$ is
defined in a similar manner to the definition in Lemma \ref{prob}.
Hence,
$$\Delta_k:=\max_{1 \le i \ne
j \le L_k}|W_k(i,j)|\le 4e^{-\frac{\delta k }{d+1}}=4L_k^{-1}$$
while Condition 4 implies that
$$M_k:=\sum_{s \le i \le L_k}\proba(A_{i,k})\ge c(L_k-s).$$
If $k \ge \frac{d+1}{\delta}\log(2s )$ then $L_k-s \ge
\frac{1}{2}L_k$. Proposition \ref{prob} implies that for every $k
\ge \frac{d+1}{\delta}\log(2s)$:

$$\prob(w_k \in Z)\le \proba(\walks(\Sigma) \setminus \bigcup_{i \in L_k}A_{i,k})
\le \frac{L_k+{L_k}^2\Delta_k}{M_k^2}\le \frac{20}{c^2L_k} =
\frac{20}{c^2}e^{-\frac{\delta k}{d+1}}.$$
\end{proof}

We are now ready to formulate the Group Large Sieve (GLS) method.
The reader may note that Theorem B of the introduction is a special
case of the following:

\begin{thm}[GLS]\label{group large sieve} Fix $s\ge 2$. Let $\Gamma$ be a finitely
generated group and let $\Sigma$ be an admissible multi-subset. Let
$\Lambda$ be a finite index subgroup of $\Gamma$ and let
$(N_i)_{i\ge 2}$ be family of normal finite index subgroups of
$\Gamma$ which are contained in $\Lambda$. Let $Z \subseteq \Sigma$
and assume that the following conditions hold:
\begin{itemize}
\item[1.] $\Gamma$ has property-$\tau$ w.r.t the family $\{N_i\cap
N_j \mid i,j \ge s\}$.
\item[2.] There exists a constant $d$ such that $|\Gamma/N_j| \le j^d$
for every $j\ge s$.
\item[3.] $|\Lambda/N_i \cap N_j|=|\Lambda/N_i||\Lambda/N_j|$ for every distinct $i,j \ge s$.
\item[4.] There exists $c>0$ such that for every coset $C\in \Gamma/\Lambda$ and every $j \ge s$,
$$|(Z\cap C)N_j/N_j|\le (1-c)|\Lambda/N_j|.$$
\end{itemize}
Then there are positive constants $\gamma$ and $t$ such that
$\prob(w_k \in Z)\le e^{-\gamma k}$ for every $k \ge t \log s$. In
particular, $Z$ is exponentially small.

In fact, the constants $\gamma$ and $t$ depends only on the spectral
gap of the family of Cayley graphs $\{\Cay(\Gamma/N_i,\Sigma
N_i/N_i) \mid i \ge 2\}$, the size of $\Sigma$, the length of the
shortest odd cycle in $\Cay(\Gamma,\Sigma)$ and the two constants
$c$ and $d$
\end{thm}
\begin{proof} It is enough to verify the Conditions of Theorem \ref{sieve theorem}.
The existence of constants $\delta$ and $d$ for which Conditions 1
and 2 of Theorem \ref{sieve theorem} hold follows from Conditions 1
and 2 of the current corollary together with Corollary \ref{property
tau corol}. Condition 4 of the corollary allows us to choose for
every $i \ge s$ a constant $b_i \ge c$ and a subset $T_i \subseteq
\Gamma/N_i \setminus ZN_i/N_i$ such that $|T_i \cap
C/N_j|=b_i|\Lambda/N_i|$ for every coset $C \in \Gamma/\Lambda$.
Conditions $0$ and $4$ of Theorem \ref{sieve theorem} are readily
satisfied by the definition of $T_i$. The definition also implies
that for every $i \ge s$, $\frac{|T_i|}{|\Gamma_{i,i}|}=b_i$. In
turn, Condition 3 of the corollary ensures that for every distinct
$i,j \ge s$, $\frac{|\{gN_{i,j} \mid gN_i \in T_i \ \wedge \ gN_j
\in T_j \}|}{|\Gamma_{i,j}|} = b_ib_j$. Thus, Condition 3 of Theorem
\ref{sieve theorem} also holds.
\end{proof}

Our next goal is to apply the group large sieve method in the
situation where $\Gamma$ is a subgroup of $\SL_n(\Z)$. For every
prime $p$, let $\pi_p:\Gamma \rightarrow \SL_n(\Z/p\Z)$ be the
modulo-$p$ homomorphism. We shall see that it is fruitful to define
$N_i:=\ker \pi_{p_i}$ where $p_2,p_3,\ldots$ is an ascending
enumeration of the primes which belong to some arithmetic
progression. In order to verify Condition 2 of Corollary \ref{group
large sieve} with respect to the sequence $(N_i)_{i \ge 2}$, we need
the next theorem about the density of the primes in arithmetic
progressions (for a proof see Chapter 22 of \cite{Da}).

\begin{thm}[Quantitative Dirichlet's Theorem]\label{qdt}
Let $0 \le \varepsilon<1$ and $t \ge 1$ be constants. Then there
exists a constant $r>0$ such that for every $k \ge r$, every $x \ge
e^{\frac{k}{t}}$ and every two coprime numbers $1 \le a,b \le k^t$
the number of primes $p$ which satisfy:
\begin{itemize}
\item $1 \le p \le  x$
\item $p$ belongs to the series $(a+bj)_{j\ge 1}$.
\end{itemize}
is at least $\frac{1-\varepsilon}{\varphi(b)}\frac{x}{\log x}$ where
$\varphi$ is the Euler function.
\end{thm}

A straightforward Corollary of Theorem \ref{qdt} is:
\begin{corol}\label{qdl}
Let $t >0$. Then there is a positive constant $r$ depending only on
$t$ such that for every $k \ge r$ and every two coprime natural
numbers $a,b \le k^t$, if $(p_i)_{i\ge 1}$ is an ascending
enumeration of the primes which belong to the sequence $(a+bi)_{i
\ge 1}$ then $p_i \le i^2$ for every $i \ge e^{\frac{k}{t}}$.
\end{corol}

We can now deduce:

\begin{prop}\label{arithmetic} Fix $n \ge 2$ and $c>0$. Let $\Gamma$ be a Zariski-dense
subgroup of $\SL_n(\Z)$ with an admissible multi-subset $\Sigma$.
For every $q \in \N^+$, let $\pi_q:\Gamma \rightarrow \SL_n(\Z/q\Z)$
be the modulo-$q$ homomorphism. Then, there exist two positive
constants $\gamma$ and $r$, depending only on $\Gamma$, $\Sigma$,
and $c$ such that the following statement holds:

Let $Z$ be a subset of $\Gamma$ and let $k \ge r$ be a natural
number. Assume that there are two coprime natural numbers $2 \le a,b
\le k^2$ such that for every prime $p$ which belongs to the
arithmetic progression $(a+bj)_{j \ge 1}$, the size of $\pi_p(Z)$ is
at most $(1-c)|\SL_n(\Z/p\Z)|$. Then,
$$\prob(w_k \in Z) \le
e^{-\gamma k}.$$
\end{prop}
\begin{proof} Let $a$ ,$b$, $c$, $k$ and $Z$ be as in the
statement. We need to show that there are positive constants
$\gamma$ and $r$ depending only on $\Gamma$ ,$\Sigma$ and $c$ such
that if $k \ge r$ then $$\prob(w_k \in Z) \le e^{-\gamma k}.$$

Denote $q_3:=q_1q_2$ where $q_1$ is as in Theorem \ref{SG-V} and
$q_2$ is as in Theorem \ref{Strong}. Note that $q_3$ depends only on
$\Gamma$ and $\Sigma$. Let $(p_i)_{i \ge 2}$ be an ascending
enumeration of the primes  which belong to the sequence
$(1+bq_3i)_{i \ge 1}$. Denote $N_i:=\ker \pi_{p_i}$ for every $i \ge
2$. Theorem \ref{SG-V} implies that $\Gamma$ has property-$\tau$
with respect to the family $(N_i \cap N_j)_{i,j \ge 2}$ and Theorem
\ref{Strong} (Strong Approximation) implies that if $q \in \N^+$ is
a product of primes which belong to $(a+bq_3i)_{i \ge 1}$ then
$\pi_q(\Gamma)=\SL_n(\Z/q\Z)$.

We now verify the four conditions of GLS (Theorem \ref{group large
sieve} for $\Lambda=\Gamma$). The above paragraph together with the
fact that $\SL_n(\Z/p_1p_2\Z) \cong \SL_n(\Z/p_1\Z) \times
\SL_n(\Z/p_2\Z)$ for distinct primes $p_1$ and $p_2$ imply that
Conditions 1 and 3 of GLS are satisfied for every $i,j\ge 2$.
Condition 4 of GLS is true for every $j \ge 2$ by assumption.
Finally, we will show that Condition 2 of GLS holds for $d:=2n^2$.
Let $\delta$ be the spectral gap of the family
$\{\Cay(\Gamma/N_i,\Sigma N_i/N_i) \mid i \ge 2\}$. Let
$\gamma:=\gamma(\delta,|\Sigma|,c,d)$ and
$t:=t(\delta,|\Sigma|,c,d)$ be the constants of GLS for $d=2n^2$.
Corollary \ref{qdl} implies that there exists a constant $r$ such
that if $k \ge r$ and $j \ge e^{\frac{k}{2t}}$ then $p_j \le j^2  $
so $|\Gamma/N_j|\le j^{2n^2}$. We can assume that $k \ge r$ and
denote $s:=e^{\frac{k}{2t}}$. Hence, Condition 2 of GLS holds for
every $j \ge s$. We finished to verify all the conditions of GLS.
Since $k \ge t \log s$, GLS implies that $\prob(w_k \in Z)\le
e^{-\gamma k}$.
\end{proof}

\subsection{An extension of Corollary \ref{arithmetic}}

This subsection is needed for the proof of the general case of
Theorem A. A reader that is only interested in the special case
(Zariski-dense subgroup of $\SL_n(\Z)$) can skip this subsection.

\begin{prop}\label{structre} Let $\Gamma \le \GL_n(\Q)$
be a finitely generated group such that its Zariski-closure in
$\GL_n(\C)$ is semisimple. Then there are:
\begin{itemize}
\item A finite index normal subgroup
$\Lambda$ of $\Gamma$.
\item A set $\mathcal{P}$ which contains almost
all primes.
\item For every prime $p\in \mathcal{P}$, an epimorphism $\pi_p:\Gamma\rightarrow
\Gamma_p$ such that $\Gamma_p$ is a non-trivial finite group and
$N_p:=\ker \pi_p $ is contained in $\Lambda$.
\item Constants $d,l,s\in \N^+$.
\end{itemize}
such that:
\begin{itemize}
\item[$1.$] $\Gamma$ has property-$\tau$ with respect to the
family $\{N_{p,q}\}_{p,q \in \mathcal{P}}$ where $N_{p,q}:=N_p\cap
N_q$.
\item[$2.$]  $\Lambda_p:=\pi_p(\Lambda)$ is a direct product of at most $l$ finite simple groups of Lie type of rank at most $l$
over finite extensions of $\F_p$ of degree at most $l$ for every $p
\in \mathcal{P}$. Furthermore, $\frac{1}{s}p^d \le |\Lambda_p| \le
sp^d \le p^{d+1}$.
\item[$3.$] The natural homomorphism $\pi_{p,q}:\Lambda_{p,q}
\rightarrow \Lambda_p \times \Lambda_{q}$ is an isomorphism for all
distinct $p,q\in\mathcal{P}$ where $\Lambda_{p,q}:=\Lambda/N_{p,q}$.
\end{itemize}
\end{prop}

\begin{proof}  Since $\Gamma$ is finitely generated it is contained in $\GL_n(\Z[\frac{1}{q_0}])$
for some $q_0 \in \N^+$. The connected component of the
Zariski-closure of $\Gamma$ in $\GL_n(\C)$, denoted by
$\mathrm{G}(\C)$, is defined over $\Q$. Hence, for every large
enough prime $p$ the group $\mathrm{G}(\F_p)$ and the residue map
$\Gamma^\circ \rightarrow \mathrm{G}(\F_p)$  is defined where
$\Gamma^\circ:=\Gamma \cap \mathrm{G}(\C) $.

We would like to apply the Strong Approximation Theorem to conclude
that for every large enough prime $p$ the residue map $\Gamma^\circ
\rightarrow \mathrm{G}(\F_p)$ is an epimorphism. But we have a
problem, our $\mathrm{G}(\C)$ is connected and semisimple but not
necessarily simply connected. To overcome this problem we let
$\psi:\tilde{\mathrm{G}}\rightarrow \mathrm{G}$ be the universal
cover of $\mathrm{G}$. The algebraic group $\tilde{\mathrm{G}}$ and
the rational homomorphism $\psi$ are defined over $\Q$. The group
$\psi(\tilde{\mathrm{G}}(\Q))$ is a normal coabelian subgroup of
$\mathrm{G}(\Q)$. This implies that there exists
$\tilde{\Gamma}_1\le \tilde{\mathrm{G}}(\Q) $ where
$\psi|_{\tilde{\Gamma}_1}$ is an isomorphism onto a finite index
subgroup $\Gamma_1$ of $\Gamma^\circ$ (see chapter 16 of
\cite{LuSe}). The image of $\Gamma_1$ under the residue map
$\Gamma_1 \rightarrow {\mathrm{G}}(\F_p)$ is the same as the image
of $\tilde{\Gamma}_1$ under the composition $\tilde{\Gamma}_1
\rightarrow \tilde{\mathrm{G}}(\F_p) \rightarrow \mathrm{G}(\F_p)$.
The Strong Approximation Theorem (Theorem \ref{Strong}) says that
for large enough prime $p$ the first homomorphism is an epimorphism.
The kernel of the second homomorphism is contained in the center of
$\tilde{\mathrm{G}}(\F_p)$ and the image is of index at most $b$ in
$\mathrm{G}(\F_p)$ where $b\in \N^+$ is some constant independent of
$p$.

We define $\Lambda$ to be the intersection of all subgroups of index
at most $b$ of $\Gamma_1$. For large enough prime $p$ we get a
homomorphism $\pi_p:\Lambda \rightarrow
\tilde{\mathrm{G}}(\F_p)/\mathrm{Z}(\tilde{\mathrm{G}}(\F_p))$ and
we denote $N_p:=\ker \pi_p$. The structure of
$\tilde{\mathrm{G}}(\F_p)/\mathrm{Z}(\tilde{\mathrm{G}}(\F_p))$ is
well known and there are $c,d,l\in \N^+$ such that the requirements
of condition 2 are satisfied (see \cite{JKZ} and the reference
therein). In fact, $d:=\dim(\tilde{\mathrm{G}}(\C))$. In particular,
this structure assures that $\pi_p$ is an epimorphism for a large
enough prime $p$ so condition 2 is satisfied for the set
$\mathcal{P}$ consisting of large enough primes. Condition 3 then
follows since two finite simple Lie groups over field of different
characteristics are not isomorphic. Condition 1 follows from Theorem
\ref{SG-V}.
\end{proof}

The proof of the next Proposition is identical to the proof of
Proposition \ref{arithmetic} so we omit it.
\begin{prop}\label{arithmetic2} Fix a positive constant $c$. Let $\Gamma$ be
as in Proposition \ref{structre} and let $\Sigma$ be an admissible
generating multi-set of $\Gamma$. Then, there exist two positive
constants $\gamma$ and $r$, depending only on $\Gamma$ and $\Sigma$
and $c$ such that the following statement holds:

Let $Z$ be a subset of $\Gamma$ and let $k \ge r$ be a natural
number. Assume that there are two coprime natural numbers $2 \le a,b
\le k^5$ such that for every prime $p$ which belongs to the
arithmetic progression $(a+bj)_{j \ge 1}$ and every coset $C \in
\Gamma_p/\Lambda_p$, the size of $\pi_p(Z)\cap C$ is at most
$(1-c)|\Lambda_p|$. Then,
$$\prob(w_k \in Z) \le
e^{-\gamma k}.$$
\end{prop}

\section{Proof of Theorem A for a Zariski-dense subgroups of $\SL_n(\Z)$}\label{chap special}

In this section $\Gamma$ denotes  a Zariski-dense subgroup of
$\SL_n(\Z)$ and $\Sigma $ is an admissible generating multi-subset
of $\Gamma$.

\begin{lemma}\label{powers lemma} Fix $m \ge 2$. Let $p$ be a prime which satisfies:
\begin{itemize}
\item $p\ge 3\binom{n}{2}+1$.
\item $p=1(\Mod m)$.
\end{itemize} Then $|\{g^m \mid g\in\SL_n(\Z/p\Z) \}| \le
(1-\frac{1}{6n!})|\SL_n(\Z/p\Z)|$.
\end{lemma}
\begin{proof}
The subset $T$ of diagonal matrices of $\SL_n(\Z/p\Z)$  is a
subgroup isomorphic to $C_{p-1}^{n-1}$ where $C_{p-1}$ is a cyclic
group of order $p-1$. Since $m$ divides $p-1$ the map $x \mapsto
x^m$ is $m^{n-1}$-to-$1$ on $T$ so
$$\left|\left\{t^m \mid t \in T\right\}\right| \le
\frac{1}{m^{n-1}}|T|\le \frac{1}{2}|T|.$$

An element of $T$ is called regular if all its non-zero entries are
distinct. The size of the set $S$ consisting of regular elements is
at least $(p-1)^{n-1}-\binom{n}{2}(p-1)^{n-2}$. Thus, if $p \ge
1+3\binom{n}{2}$ then $$|S| \ge \frac{2}{3}|T|$$ while
$$\left|\left\{t^m \mid t \in S\right\}\right| \le \frac{1}{2}|T|.$$

Let $s \in S$, then the centralizer of $s$ is $T$. If $g \in
\SL_n(\Z/p\Z)$ and $gsg^{-1}\in T$ then $gsg^{-1}$ is also regular
so its centralizer is also $T$. On the other hand, conjugation by
$g$ is an automorphism, so the centralizer of $gsg^{-1}$ is
$gTg^{-1}$. Thus, $T=gTg^{-1}$ and $g$ belongs to the normalizer $N$
of $T$. The normalizer $N$ is the set of monomial matrices, so $N/T$
is isomorphic to the symmetric group on $n$ elements and $[N:T]=n!$.

Let $R$ be a set of representatives of the left cosets of $N$. Then,
$$|R|=\frac{1}{n!}[\SL_n(\Z/p\Z):T].$$ If
$r_1,r_2\in R$ are distinct then $r_1Sr_1^{-1}$ and $r_2Sr_2^{-1}$
are disjoint. Thus for $\bar{S}:=\cup_{r \in R} rSr^{-1}$,
$$|\bar{S}|\ge \frac{2}{3}|T|\cdot\frac{1}{n!}\frac{|\SL_n(\Z/p\Z)|}{|T|}
= \frac{2}{3n!}|\SL_n(\Z/p\Z)|$$ while $$|\{s^m \mid s\in
\bar{S}\}|\le
\frac{1}{2}|T|\cdot\frac{1}{n!}\frac{|\SL_n(\Z/p\Z)|}{|T|}=
\frac{1}{2n!}|\SL_n(\Z/p\Z)|.$$
\end{proof}

Combining Lemma \ref{powers lemma} with corollary \ref{arithmetic}
we get:
\begin{corol}\label{power corol} There exist two positive constants $\gamma$ and $r$ such that if
$k,m\in \N^+$ with $k \ge r$ and $2 \le m \le k^2$ then $$\prob(w_k
\in \Gamma^m) \le e^{-\gamma k}.$$
\end{corol}

The next lemma shows what kind of $m$-powers are possible at the
$k^{\text{th}}$-step of a random walk.
\begin{lemma}\label{two parts}
There is a constant $s \in \mathbb{N}^+$ such that if the
$k^{\text{th}}$-step $w_k$ of a random walk on $\Cay(\Gamma,\Sigma)$
is a proper power then one of the following holds:
\begin{itemize}
\item $w_k$ is virtually-unipotent.
\item $w_k=g^m$ for some element $g\in G$ and some prime number $m\le sk$.
\end{itemize}
\end{lemma}
\begin{proof}
For $x \in \C^n$ let $|x|$ be $L^2$-norm of $x$. Recall that the
operator norm of an element $g \in \Gamma$ is
$||g||:=\max_{|x|=1}|gx|$. Note that if $g,h \in \Gamma$ and
$\lambda$ is an eigenvalue of $g$ then $||gh|| \le ||g||||h||$ and
$||\lambda|| \le ||g||$. Define $c:=\max_{g \in \Sigma}||g||$ so
$||w_k|| \le c^k$ for every walk $w$.

If a polynomial $f$ of degree $n$ with integer coefficients is not a
product of cyclothymic polynomials then it has a root with absolute
value greater then $1+\varepsilon$ where $\varepsilon$ depends only
on $n$ (see for example Proposition 5.5 and Corollary 5.6 of
\cite{Mi}). Thus, if $g \in \SL_n(\Z)$ is not virtually unipotent
then it has an eigenvalue $\lambda$ with absolute value greater then
$1+\varepsilon$ so $||g^m||\ge (1+\varepsilon)^m$ for every $m\in
\N$.

A power of a virtually unipotent element is virtually unipotent.
Hence, if the the $k^{\text{th}}$-step $w_k$ of a walk is an
$m$-power but not virtually unipotent then $(1+\varepsilon)^m \le
c^k$ so $m \le sk$ for $s:=\frac{\log c}{log (1+\varepsilon)}$.
\end{proof}

Our goal is to show that the set of proper powers in $\Gamma$ is
exponentially small. The set of virtually unipotent elements in
$\Gamma$ is exponentially small by Corollary \ref{unipotent}. Thus,
Lemma \ref{two parts} implies that it is enough to show that the set
$Z \subseteq \Gamma$ of elements which are proper powers but not
virtually unipotent is exponentially small. Let $s$ be as in Lemma
\ref{two parts} and let $\gamma$ and $r$ be as in \ref{power corol}.
If $k \ge \max (r,s)$ then
$$\prob(w_k \in Z)=\prob(w_k \in \cup_{2
\le m \le k^2} \Gamma^m) \le k^2e^{-\gamma k}.$$  If $\alpha$ is a
positive constant smaller than $\gamma$ then for large enough $k$,
$$\prob(w_k \in Z) \le k^2e^{-\gamma k} \le e^{-\alpha k}.$$
Thus, $Z$ is indeed exponentially small and the proof is complete.

\section{Proof of Theorem A}\label{gen}

\subsection{Reduction}

The next lemma shows that it is enough to prove Theorem A for a
finitely generated subgroup of $\GL_n(\Q)$ such that the connected
component of its Zariski-closure is a non-trivial semisimple group.
\begin{lemma}\label{reduction}
Let $\Gamma \le \GL_m(\C)$ be a finitely generated group which is
not virtually solvable. Then there is a positive integer $n$ and a
homomorphism $\alpha:\Gamma \rightarrow \GL_n(\Q)$ such that the
connected component of the Zariski-closure of $\alpha(\Gamma)$ is
semisimple.
\end{lemma}

\begin{proof} Proposition $16.4.13$ of \cite{LuSe} shows that there is
an $n_1\in \N^+$ and a homomorphism $\alpha_1:\Gamma \rightarrow
\GL_{n_1}(\Q)$ such that $\alpha(\Gamma)$ is not a virtually
solvable group.  The connected component of the Zariski-closure of
$\alpha(\Gamma)$ is not necessarily semisimple. However, we can
divide it by its solvable radical, which is defined over $\Q$. Thus,
there is $n_2\in \N^+$ and a homomorphism $\alpha_2:\alpha_1(\Gamma)
\rightarrow \GL_{n_2}(\Q)$ such that the connected component of the
Zariski-closure of $\alpha_2\circ\alpha_1(\Gamma)$ is semisimple.
Define $\alpha:=\alpha_2\circ\alpha_1$ and $n:=n_2$.
\end{proof}

\subsection{Virtually unipotent elements}

We start this section with an analog of Proposition \ref{subvarity}.

\begin{prop}\label{11111} Let $\Gamma$ be a finitely generated subgroup of $\GL_n(\Q)$
whose Zariski-closure $\bar{\Gamma}$ is semisimple. Assume that
$V(\C)$ is a variety defined over $\Q$ and that $V(\C)$ does not
contain any coset of the connected component of $\bar{\Gamma}$.
Then, $V(\C)\cap \Gamma$ is exponentially small.
\end{prop}

The proof of Proposition \ref{11111} is almost identical to the one
of Proposition \ref{subvarity} so we omit it. The next Corollary is
the main result of this subsection.

\begin{corol}\label{234454} Let $\Gamma$ be a finitely generated subgroup of $\GL_n(\Q)$
whose Zariski-closure is semisimple. Then, the set of virtually
unipotent elements is exponentially small.
\end{corol}
\begin{proof}
As in the proof proof of Corollary \ref{unipotent}, there exists a
positive integer $t$ such that if $g \in \Gamma$ is virtually
unipotent then $g^t$ is unipotent. Thus, the set of virtually
unipotent elements is contained in a subvariety defined over $\Q$.
Proposition \ref{subvarity} above and Proposition \ref{proper
subvariety lemma} below complete the proof.
\end{proof}

\begin{lemma}\label{proper subvariety lemma} Let $t \in \N^+$.
Let $\Gamma$ be a finitely generated subgroup of $\GL_n(\Q)$ whose
Zariski-closure in $\GL_n(\C)$, denoted by $\bar{\Gamma}$, is a
semisimple group. Then every coset of the identity component
$\bar{\Gamma}^\circ$ of $\bar{\Gamma}$ contains an element whose
$t$-power is not unipotent.
\end{lemma}
\begin{proof} By replacing $\Gamma$ with its image in
$\bar{\Gamma}^\circ/\mathrm{Z}(\bar{\Gamma}^\circ)$ we can assume
that $\bar{\Gamma}^\circ$ has trivial center. Let $C$ be some coset
of $\bar{\Gamma}^\circ$ and assume that for every $g \in C \cap
\Gamma$ the power $g^t$ is unipotent. Then the eigenvalues of every
element of $C \cap \Gamma$ are roots of unity of bounded order.
Thus, there is only a finite number of Jordan forms for the elements
of $C \cap \Gamma$. The set of elements with a given Jordan form is
Zariski-closed so all the elements of $C$ have the same Jordan form
since $C$ is irreducible and $C \cap \Gamma$ is dense in it. In
particular, all elements of $C$ have the same order.

The automorphism group of $\bar{\Gamma}^\circ$ is a semidirect
product of the group of inner automorphism and the finite group of
outer automorphisms. Hence, every coset of
$\inn(\bar{\Gamma}^\circ)$ in $\aut(\bar{\Gamma}^\circ)$ contains an
element of finite order. This implies that there are some $k \in
\N^+$ and $g \in C$ such that $g^k \in
\cen_{\bar{\Gamma}}(\bar{\Gamma}^\circ)$. However, this centralizer
is finite so $g$ has finite order. Thus, all the elements of $C$ has
the same finite order. In the next paragraph  we will show that the
orders of the elements in every coset of $\inn(\bar{\Gamma}^\circ)$
are unbounded so the same is true for $C$ and we get the desired
contradiction.

If $\bar{\Gamma}^\circ$ is a simple group of adjoint type then every
coset of $\inn(\bar{\Gamma}^\circ)$ in $\aut(\bar{\Gamma}^\circ)$
contains a graph automorphism $\alpha$. The graph automorphism
$\alpha$ fixes some root of the Dynkin diagram unless the diagram is
of type $A_{2n}$ and in that case the automorphism switches between
the two roots at the ends. In any case, $\alpha$ pointwise fixes
some torus $T$ and the orders of the elements of $\alpha T \subseteq
\alpha\inn(\bar{\Gamma}^\circ)$ are unbounded (we identified $T$
with its image in $\inn(\bar{\Gamma}^\circ)$).

In the general case $\bar{\Gamma}^\circ$ has a characteristic
subgroup $N$ such that $\bar{\Gamma}^\circ/N$ is isomorphic to
$\Lambda^k$ where $\Lambda$ is a simple group of adjoint type and $k
\in \N^+$. It suffices to show that for every $\alpha \in
\aut(\Lambda^k)$ the orders of the elements which belong to $\alpha
\inn (\Lambda^k)$ are unbounded. As before, this will follow once we
show that there is an element in $\alpha \inn (\Lambda^k)$ which
pointwise fixes a non trivial torus. As $\alpha \in
\aut(\Lambda^k)$, there are $\alpha_1,\cdots,\alpha_k \in
\aut(\Lambda)$ and a permutation $\sigma\in \textrm{Sym}(k)$ such
that for every $(x_1,\cdots,x_k)\in \Lambda^k$ we have
$$\alpha(x_1,\cdots,x_k)=(\alpha_1(x_{\sigma(1)}),\cdots,\alpha_k(x_{\sigma(k)})).$$

If the type of $\Lambda$ is different from $D_4$ then there is only
one non-identity graph automorphism, so by the previous paragraph we
see that there is a non-trivial torus $T$ which is pointwise fixed
by all graph automorphisms. If the type of $\Lambda$ is $D_4$ then
all the graph automorphisms fix the central root and thus pointwise
fix the torus $T$ corresponds to this root. By replacing $\alpha$
with another representative of the coset $\alpha\inn(\Lambda^k)$ we
can assume that $\alpha_i$ is a graph automorphism for every $1 \le
i \le k$ so $\alpha$ pointwise fixes $T^*:=\{(t,\cdots,t) \mid t \in
T\}$. As before the fact the the orders of $\alpha T^*$ are
unbounded implies that the orders of $\alpha \inn (\Lambda^k)$ are
unbounded.
\end{proof}

\subsection{Completion of the proof of Theorem A}

Let $\Gamma$ be a finitely generated subgroup of $\GL_n(\Q)$ whose
Zariski-closure $\bar{\Gamma}$ is semisimple. Then, there are
finitely many primes $p_1,\ldots,p_r$ such that $\Gamma$ is
contained in $\GL_n(\Z[\frac{1}{p_1\cdots p_r}])$. Fix an admissible
generating  multi-subset $\Sigma$ of $\Gamma$.

We start with some number theoretic arguments. For every $1 \le i
\le r$, let $|\cdot|_{i}$ be some extension of a $p_i$-adic
valuation of $\Q$ to the algebraic closure $\tilde{Q}$ of $\Q $. Fix
an embedding of $\tilde{Q}$ in $\C$ and let $|\cdot|_0$ be the
restriction  to $\tilde{\Q}$ of the absolute value of $\C$. It is
well known that for every $n\in\mathbb{N}^+$, there exists a
constant $c>1$ such that if $x $ is an algebraic integer of degree
at most $n$ then either $x$ is a root of unity or some Galois
conjugate  $y$ of $x$ satisfies  $|y|_0 \ge c$ (see Proposition 5.5
and Corollary 5.6 of \cite{Mi}). The following lemma is a
straightforward generalization of this fact.
\begin{lemma}\label{valuation lemma} Denote $R:=\Z[\frac{1}{q}]$
where $q:=p_1\cdots p_r$.  Then, for every $n\in\mathbb{N}^+$ there
exists a constant $c>1$ such that if $x \in \tilde{Q}^*$ is integral
over $R$ of degree at most $n$ then either $x$ is a root of unity or
$|y|_i\ge c$ for some $0 \le i \le r$ and some Galois conjugate $y$
of $x$.
\end{lemma}
\begin{proof} Let $x\in \tilde{Q}^*$ be integral over
$R$ such that its minimal polynomial $f$ over $R$ has degree at most
$n$. The roots of $f$ are the Galois conjugates of $x$. The case
where all the coefficients of $f$ belong to $\Z$ is the classical
case above. Thus, we can assume that some coefficient $b$ of $f$
does not belong to $\Z$. Define $d:=\min_{1\le i \le
r}|{p_i}|_{i}^{-\frac{1}{n}}$ and note that $d>1$. There is $1 \le j
\le r$ such that $p_j$ is a factor of the denominator of $b$ so
$|b|_j\ge d^n$. The coefficient $b$ is a symmetric polynomial of
degree at most $n$ in the Galois conjugates of $x$ and $|\cdot|_j$
is non-archimedean so there is at least one Galois conjugate $y$ of
$x$ with $|y|_j \ge d$.
\end{proof}

For $1 \le i \le r$ and $\bar{x}=(x_1,\cdots,x_n) \in \tilde{Q}^n$
define $$|\bar{x}|_i:=\max_{1 \le j \le n}|x_j|_i$$ and
$$|\bar{x}|_0:=\sqrt{\sum_{j=1}^r |x_j|_0^2}.$$
For an element $g \in \Gamma$ and $0 \le i \le r$, define:
$$|g|_i:=\max_{\bar{x}\ne 0}\frac{|g\bar{x}|_i}{|x|_i}.$$
Note that if $g,h \in \Gamma$ and $\lambda$ is an eigenvalue of $g$
then $|gh|_i \le |g|_i|h|_i$ and $|\lambda|_i \le |g|_i$.

\begin{lemma}\label{two parts2}
There is a constant $t \in \mathbb{N}^+$ such that for every
$k\in\mathbb{N}^+$ and every walk $w$ on $\Cay(\Gamma,\Sigma)$, if
$w_k$ is a proper power then one of the following holds:
\begin{itemize}
\item $w_k$ is virtually-unipotent.
\item $w_k=g^m$ for some element $g\in G$ and some prime number $m\le tk$.
\end{itemize}
\end{lemma}
\begin{proof} Fix $k \in \mathbb{N}$ and define:
$$b:=\max_{g \in \Sigma\ \wedge\ 0 \le i \le r}|g|_i.$$

Then, $|w_k|_i \le b^k$ for every walk $w$ and every $0 \le i \le
r$. Let $g\in \Gamma$. If $\lambda$ is an eigenvalue of $g$ then so
are all the Galois conjugates of it. Lemma \ref{valuation lemma}
shows that there is a constant  $c>1$ such that if $g$ is not
virtually unipotent then $|\lambda|_i \ge c$ for some eigenvalue
$\lambda$ of $g$ and some $0 \le i \le r$. In the later case for
every $m\in \N^+$ we have that $\lambda^m$ is an eigenvalue of $g^m$
and $|g^m|_i \ge |\lambda^m|_i \ge c^m$. Thus, if $w$ is a walk on
$\Cay(\Gamma,\Sigma)$ and $w_k=g^m$ for some $g\in \Gamma$ and some
$m\in \N^+$ then either $w_k$ is virtually unipotent or $m \le tk$
where $t:=\frac{\log b}{\log c}$.
\end{proof}

\begin{lemma}\label{2434342323}  There exist two positive constants $\alpha$ and $r$ such that if
$k\ge r$  then $$\prob(w_k \in \cup_{2 \le m \le k^2}\Gamma^m) \le
e^{-\alpha k}.$$
\end{lemma}
\begin{proof} We use the notation of Proposition \ref{structre}. For every
coset $D \in \Gamma/\Lambda$ and every $m\ge 2$ define $D^m:=\{g^m
\mid g \in D\}$. Proposition \ref{arithmetic2} together with
Corollary \ref{done} imply that there are positive constants
$\gamma$ and $r$ such that for every $k \ge r$ and for every coset
$D \in \Gamma/\Lambda $:
$$\prob(w_k\in \cup_{2 \le m \le k^2}D^m) \le k^2e^{-\gamma k}.$$
The subgroup $\Lambda$ is of finite index in $\Gamma$ so it has only
finitely many cosets. Hence, if $0 < \alpha < \gamma$ and $k$ is
large enough then:
$$\prob(w_k \in \cup_{2 \le m \le k^2}\Gamma^m) \le
e^{-\alpha k}.$$
\end{proof}

Corollary \ref{234454} together with Lemmas \ref{two parts2} and
\ref{2434342323} complete the proof of Theorem A (modulo the proof
of Theorem \ref{power main4} and corollary \ref{done}).

\section{Powers in finite groups of Lie Type}\label{chap liegroups}

The goal of this section to bound the number of powers in a finite
extension of a direct product of finite groups of Lie type (see
Theorem \ref{power main4} below). This bound is needed for the proof
of Theorem \ref{power theorem} but it might be interesting on its
own right.

\subsection{Notation}

The letter $p$ always denotes a prime number greater than $3$
(sometimes further restrictions which will be specified). The letter
$q$ denotes some power of $p$. For a group $G$, a subgroup $T
\subseteq G$ and an automorphism $\mu \in \aut(G)$, we denote
$T_\mu:=\{t \in T \mid \mu(t)=t\}$. In most cases the group $T$ will
be abelian and $\mu$ will preserve $T$. In addition, if $n \in \N$
then $T_n:=\{t \in T \mid \cen_G(t^n)=T\}$ and $T_{\mu,n}:=T_\mu
\cap T_n$.

We fix some root system $\Phi \subseteq\mathbb{R}^l$ and let
$W(\Phi)$ be its Weyl group. The symbols $G_q$ ($G_q^*$) represents
the Simply connected Chevalley (Steinberg) group of type $\Phi$ over
the field with $q$ elements $\F_q$. Unless otherwise mentioned, the
term finite Lie group means a Chevalley or Steinberg group of
adjoint type although the results apply to all finite Lie groups. A
detailed description of these groups is given in later sections.

\subsection{Powers in arbitrary groups}\label{genral setting}

We start this section with three lemmas.
\begin{lemma}\label{centralizer lemma} Let $G$ be a finite
group with a maximal abelian subgroup $T$. Assume $\mu \in \aut(G)$
preserves $T$ and denote $n:=\ord(\mu)$. If $t \in T_{\mu,n} $ then
$$T_\mu=\{g\in G \mid gt\mu(g)^{-1}=t\}.$$
\end{lemma}
\begin{proof}
The inclusion $\subseteq$ is clear. For the reverse inclusion assume
that $gt=t\mu(g)$. The element $t$ is fixed by $\mu$, so
$\mu^i(g)t=t\mu^{i+1}(g)$ for every $0 \le i \le {n}-1$. Hence,
$$gt^{n}=t\mu(g)t^{{n}-1}=t^2\mu^2(g)t^{{n}-2}=\cdots=t^n\mu^{n}(g)=t^{n}g.$$
This means the $g$ centralize $t^{n}$ so $g \in T$. Thus,
$gt\mu(g)^{-1}=t=gtg^{-1}$ so $\mu(g)=g$, i.e $g \in T_\mu$.
\end{proof}
\begin{lemma}\label{mormalizer lemma} Let $G$ be a finite group with a maximal abelian subgroup $T$.
Let $g \in G$ and $t \in T$. If $\cen_G(t)=T$ and $gtg^{-1}\in T$
then $g\in N$ where $N:=N_G(T)$.
\end{lemma}
\begin{proof} Conjugation by $g$ is an automorphism so
$\cen_{G}(gtg^{-1})=g\cen_{G}(t)g^{-1}=gTg^{-1}$. But $T$ is abelian
and so $T\subseteq \cen_{G}(gtg^{-1})=gTg^{-1}$ which implies
$T=gTg^{-1}$ and $g\in N$.
\end{proof}
\begin{lemma}\label{twist normalizer lemma} Let $G$ be a finite
group with a maximal abelian subgroup $T$. Assume $\mu \in \aut(G)$
preserves $T$ and denote $n:=\ord(\mu)$. Let $c$ be the number of
elements of $T$ of order dividing  $n$. If
$|T_{\mu,n}|>\frac{1}{2}|T_\mu|$, then the set
$$L:=\{g \in G \mid \exists s\in T_{\mu,n} \text{ s.t.
} gs\mu(g^{-1})\in T_{\mu,n}\}$$ is a subgroup of $G$ which contains
$T_\mu$ as a subgroup of index at most $c[N:T]$ where $N:=N_G(T)$.
\end{lemma}
\begin{proof}
Let $g\in L$. We start by showing that $gt\mu(g^{-1})\in T_\mu$ for
every $t\in T_\mu$. Choose $s\in T_{\mu,n}$ such that
$gs\mu(g^{-1})\in T_{\mu,n}$. Every element of $T_\mu$ is fixed by
$\mu$, so
$$gs\mu(g^{-1})=\mu(g)s\mu^2(g^{-1})\text{  and  }
\mu(g^{-1})gs\mu(g^{-1})\mu^2(g)=s.$$ Lemma \ref{centralizer lemma}
together with the fact that
$\mu(\mu(g^{-1})g)^{-1}=\mu(g^{-1})\mu^2(g)$ implies that
$u:=\mu(g^{-1})g \in T_\mu$. The element $u$ is fixed by $\mu$ so
$$u=\mu^i(u)=\mu^{i+1}(g^{-1})\mu^i(g)$$ for every $0 \le i \le
n-1$. Multiplying these equations we get that
$$u^n=\mu^{n-1}(u)\cdots \mu^{1}(u)u=1$$
which means that the order of $u$ divides $n$.

It is clear that $L$ is closed to inverses so $g^{-1}\in L$ and the
argument above assures that $v:=\mu(g)g^{-1}\in T_\mu$. Since
$gsg^{-1}v^{-1}=gs\mu(g^{-1})\in T_{\mu,n}$ also $gsg^{-1}\in
T_{\mu}$. The fact that $s \in T_{\mu,n}$ implies that
$\cen_G(s)=T$, so Lemma \ref{mormalizer lemma} shows that $g\in N$.
In particular, if $t\in T_{\mu}$ then
$gt\mu(g^{-1})=gtg^{-1}v^{-1}\in T$. In order to show that
$gt\mu(g^{-1})\in T_{\mu}$ we have to show that it is fixed by
$\mu$. Indeed,
$$\mu(gt\mu(g^{-1}))=\mu(g)t\mu^2(g^{-1})=gg^{-1}\mu(g)t\mu^2(g^{-1})\mu(g)\mu(g^{-1})=gu^{-1}tu\mu(g^{-1})=gt\mu(g^{-1})$$
where that last equality is true since  $T$ is commutative.

Define $Y:=\{g \in G \mid gT_\mu\mu(g^{-1})=T_\mu\}$, clearly $Y$ is
a subgroup of $G$. The previous two paragraphs show that $L
\subseteq Y$, in fact there is an equality. Indeed, the pigeon hole
principle together with the fact $|T_{\mu,n}|>\frac{1}{2}|T_\mu|$
show that if $g \in Y$ then $gs\mu(g)^{-1}\in  T_{\mu,n}$ for some
$s\in T_{\mu,n}$.

We have established that $L$ is a group contained in $N$ and it is
clear that $T_\mu \le L$.
 Hence, in order to show that $[L:T_\mu]\le
c[N:T]$ we only have to prove that $[L\cap T:T_{\mu}]\le c$. We
showed that if  $g \in L$ then $\mu(g)g^{-1}$ is an element of
$T_{\mu}$ whose order divides $n$. The number of elements of
$T_{\mu}$ of order dividing $n$ is $c$ so it will be enough to show
that if $t_1,t_2\in L \cap T$ and
$\mu(t_1)t_1^{-1}=\mu(t_2)t_2^{-1}$ then $t_1t_2^{-1}\in T_{\mu}$.
Indeed, since $T$ is abelian the equality
$\mu(t_1)t_1^{-1}=\mu(t_2)t_2^{-1}$ implies that
$t_1t_2^{-1}=\mu(t_1t_2^{-1})$, i.e $t_1t_2^{-1}\in T_{\mu}$.
\end{proof}

The next proposition is the one we will use in the following
sections.

\begin{prop}\label{power proportion prop} Let $m$ be a prime number.
Let $H$ be a finite group containing a normal subgroup $G$. Fix some
$h \in H$ and let $n$ be the order of the automorphism $\mu\in
\aut(G)$ induced by conjugation by $h$. Let $T$ be a maximal abelian
subgroup of $G$ preserved by $\mu$ and denote by $c$ the number of
elements of $T$ of order dividing $n$. Assume that:
\begin{enumerate}
\item[1.] $m$ divides $|T_\mu|$.
\item[2.] $|T_{\mu,n}|\ge \frac{3}{4}|T_\mu|$.
\end{enumerate}
Then, for $N:=N_G(T)$  $$|\{x^m \mid x \in Gh \}| \le
\left(1-\frac{1}{4c[N:T]}\right)|G|.$$
\end{prop}
\begin{proof} Choose an element $s \in T_\mu$ of order $m$. If $t \in T$
then $(th)^m=(tsh)^m$ since $T$ is abelian and $h$ commute with $s$.
Condition $2$ implies that the number of elements $t \in T_{\mu,n}$
such that $ts\in T_{\mu,n}$ is at least $\frac{1}{2}|T_\mu|$. This
shows that we can choose a set $A \subseteq T_{\mu,n}$ such that:
\begin{itemize}
\item The  size of $A$ is at least $\frac{1}{4}|T_\mu|$.
\item The set $B:=\{as \mid a \in A \}$ is contained in $T_{\mu,n}$.
\item $A \cap B=\emptyset$.
\end{itemize}
Choose a set of representatives $R$ for $G/L$ where $L$ is defined
in Lemma \ref{twist normalizer lemma}. If $t_1,t_2\in A \cup B$ and
$r_1,r_2 \in R$ satisfies $r_1t_1\mu(r_1^{-1})=r_2t_2\mu(r_2^{-1})$
then $t_1=t_2$ and $r_1=r_2$. The equality $rth
r^{-1}=rt\mu(r^{-1})h$ for $t\in T$ and $r\in R$ implies that
$$|\{rthr^{-1} \mid t\in A\cup B\ \wedge \ r \in
R\}|=2|A||R|$$ while
$$|\{(rth r^{-1})^m \mid x\in A\cup B \ \wedge \ r \in R\}|=
|\{(rth r^{-1})^m \mid t\in A \ \wedge\ r \ \in R\}|\le |A||R|.$$
Lemma \ref{twist normalizer lemma} assures that $|R||A|\ge
\frac{1}{4c[N:T]}|G|$ so
$$|\{x^m \mid x \in Gh \}| \le |G|-|A||R| \le
\left(1-\frac{1}{4c[N:T]}\right)|G|.$$
\end{proof}

\subsection{Chevalley groups} In this subsection we briefly describe
the Chevalley groups and their basic properties. More details and
proofs can be found in the classical book of Carter \cite{Ca}.

Let $\Phi \subseteq \mathbb{R}^l$ be an indecomposable root system
and fix a fundamental system of roots $\Pi \subseteq \Phi$. Let $\L$
be the simple Lie algebra over $\mathbb{C}$ associated with $\Phi$.
We regard $\Phi$ as a subset of $\L$ and for every $r\in \Phi$ we
define $h_r:=\frac{2r}{(r,r)}$ where $(\cdot,\cdot)$ is the usual
scalar product of $\mathbb{R}^l$. Finally we fix a Chevalley basis
$\{h_r \mid r \in \Pi\}\cup\{e_r \mid r\in\Phi\}$ of $\L$. The
multiplication of $\L$ satisfies:
\begin{enumerate}
\item[1.] $[h_rh_s]=0$.
\item[2.] $[h_re_s]=(h_r,s)e_s$.
\item[3.] $[e_re_s]=h_r$ if $r+s=0$.
\item[4.] $[e_re_s]\in \mathbb{Z}e_{s+r}$ if $r+s \in \Phi$.
\item[5.] $[e_re_s]=0$ if $r+s \not \in \Phi\cup\{0\}$.
\end{enumerate}

In particular, the product of every two elements of the Chevalley
basis is a linear combination with rational integer coefficient of
the basis elements. Thus, there is a Lie algebra $\L_q$ over $\F_q$
with the same basis as $\L$ and a similar multiplication.  The
coefficients of the product of elements of the basis in $\L_q$ are
equal modulo $p$ to the ones of the product in $\L$.

The rules of the multiplication also show that for every $r\in \Phi$
the linear map $\textrm{ad }e_r: \L \rightarrow \L$ is nilpotent and
so $\textrm{exp} (t\ \textrm{ad } e_r)$ is an automorphism of $\L$
for every $t \in \mathbb{C}$. It turns out that $\textrm{exp} (t\
\textrm{ad } e_r)$ is in fact is a finite power series of the form
$$\underset{0 \le k \le n}\sum a_k (t \textrm{ad } e_r)^k$$
where $a_k \in \mathbb{Z}$ for $0\le k\le n$. Hence,  this power
series can be evaluated also for $t \in \F_q$ and the result is an
automorphism of $\L_q$ denoted by $x_{r,t}$. The Chevalley group
$G_q$ associated to $\Phi$ is the subgroup of the automorphism group
of $\L_q$ generated by the $x_{r,t}$ for $r \in \Phi$ and $t\in
\F_q$.

In the next few paragraphs we will define certain subgroups of
$G_q$. In order to keep the notations simple we will not insert the
letter $q$ in the symbols of these subgroup but the reader should
always remember that these groups depends on the group $G_q$ which
in turn depend on $q$ and the root system $\Phi$.

For every root $r \in \Phi$ and an element $t\in \F_q^*$ there is an
element $d_{r,t}\in G_q$ which satisfies
$d_{r,t}(e_s)=t^{(h_r,s)}e_s$  and $d_{r,t}(h_s)=h_s$ for every
$s\in \Phi$. The diagonal subgroup $T$ is the group generated by the
these elements. The group $T$ is a maximal abelian subgroup and, as
the notation suggest, it will play the roll of the maximal abelian
subgroup of the first section. As before, for an $n \in \mathbb{N}$
denote $T_n:=\{t \in T \mid \cen_{G_q} (t^n)=T\}$. Our first goal is
to find a sufficient condition for an element of $T$ to belong to
$T_n$.

Every root $s \in \Phi$ can be written in a unique way as
$s=\epsilon_s\sum_{r \in \Pi}n_{s,r}r$ where $\epsilon=\pm 1$ and
the $n_r$'s are natural numbers. The set of positive roots $\Phi^+$
contains the roots $s$ with $e_s=1$ and the set of negative roots
$\Phi^-$ contains the roots $s$ with $e_s=-1$. The height of a root
$s$ is defined to be $\epsilon\sum_{r \in \Pi} n_{s,r}$. Let
$\L_q^+$ ($\L_q^-$) be the subspace of $\L$ which is generated by
the elements $e_r$ where $r$ runs over the positive (negative) roots
and let $\L_q^0$ the subspace generated by the $h_r$, with
$r\in\Pi$.

The upper unipotent subgroup $U$ of $G_q$ is the group generated by
the elements $x_{r,t}$ for $r\in\Phi^+$ and $t\in\F_q$. Similarly,
the lower unipotent subgroup $V$ of $G_q$ is the group generated by
the elements $x_{r,t}$ for $r\in\Phi^-$ and $t\in\F_q$.Our first
lemma is rater technical.
\begin{lemma}\label{the cartan subalgebra is not preserved by unipotent}
If $u \in U$ preserves $\L_q^0$ then $u$ is the identity.
\end{lemma}
\begin{proof} Let $u \in U$ be a non-identity element.
Fix an ordering $\succ$ of the roots such that $r \succ s$ implies
that the height of $r$ is greater or equal to the height of $s$.
There are unique $t_r \in \F_q$ such that
$$u=\underset{r\in\Phi^+}\prod x_{r,t_r}$$ where the product in taken
in an increasing order of the roots. Let $r^*\in \Phi^+$ be the
minimal root with respect to $\succ$ such that $t_{r^*}\ne 0$. The
definition of $x_{r,t}$ for $r\in \Phi^+$ and $t\in \F_q$ implies
that for $s\in \Phi$:
\begin{enumerate}
\item[1.] $x_{r,t}(h_s)=h_s-t_{r^*}(h_s,r)e_r$.
\item[2.] $x_{r,t}(e_s)=e_s+v$ where $v$ is a linear combination of roots
with height greater then the height of $s$.
\end{enumerate}
Thus, $u(h_{r^*})=h_{r*}-2t_{r^*}e_{r*}+v$ where $v$ is a linear
combination of roots which are greater than $r^*$ ($0 \ne 2 (\Mod
p)$ since $p\ge 5$).
\end{proof}

The last ingredient needed for the proof of Lemma \ref{reg semi
simple} is the Bruhat decomposition.  The diagonal group $T$
normalize $U$ and so $B:=TU$ is a subgroup of $G_q$. Let $N$ be the
normalizer of $T$ in $G$, $N$ is called the monomial group.  Every
$n\in N$ acts as a permutation on the set $\{\mathbb{C}e_r \mid r
\in \Phi\}$. This action defines an epimorphism of $N$ onto to Weyl
group $W(\Phi)$ with kernel $T$. For every $w \in W(\Phi)$ choose a
representative $n_w\in N$. The Bruhat decomposition says that
$$G_q=\underset{w\in W(\Phi)}\bigcup Bn_wB$$
where the union is disjoint.

\begin{lemma}\label{reg semi simple} Assume that $d\in T$ satisfies:
\begin{enumerate}
\item[1.] If $r\in \Phi$ then $d(e_r)=\lambda_r
e_r$ where $\lambda_r \ne 1$.
\item[2.] If $\lambda_r=\lambda_s$ then $r,s \in \Phi^+$ or
$r,s\in\Phi^{-}$.
\end{enumerate}
Then $T$ is the centralizer of $d$.
\end{lemma}
\begin{proof} Note that $L_q^0$ is exactly the eigen space of $d$ with
eigenvalue $1$. Assume $g\in G_q$ centralizes $d$. Conditions $1$
and $2$ imply that $g$ preserves $\L_q^+$.  Choose $b_1,b_2\in B$
and $w\in W(\Phi)$ such that $g=b_1n_wb_2$. The element
$n_w=b_1^{-1}gb_2^{-1}$ also preserves $\L_q^+$ since $B$ does.
Thus, $w(\Phi^+)=\Phi^+$ since
$n_w(\mathbb{C}e_r)=\mathbb{C}e_{w(r)}$ for every $r\in \Phi$. The
identity is the only element of the Weyl group which preserve
$\Phi^+$ so $w=\id$ and $g\in B$. Write $g=cu$ with $c \in T$ and $u
\in U$. Since $c$ centralizes $d$ also $u$ centralizes it. Condition
$1$ shows that $u$ preserves $\mathcal{L}_q^0$ and so $u=\id$ by
Lemma \ref{the cartan subalgebra is not preserved by unipotent}.
\end{proof}
\begin{corol}\label{reg semi simple-corol} Let $n \in \mathbb{N}$. Assume that $d\in T$ satisfies:
\begin{enumerate}
\item[1.] If $r\in \Phi$ then $d(e_r)=\lambda_r
e_r$ where $\lambda_r^n \ne 1$.
\item[2.] If $\lambda_r^n=\lambda_s^n$ then $r,s \in \Phi^+$ or
$r,s\in\Phi^{-}$.
\end{enumerate}
Then $d \in T_n$.
\end{corol}

Now that we have  a sufficient criterion for an element to belong to
$T_n$ we want to show that a high proportion of the element of $T$
belongs to $T_n$.

\subsection{Automorphisms of Chevalley groups} The chevalley group
$G_q$ has trivial center so it is isomorphic to the inner
automorphism group and we shell identify these two groups. In
general, $G_q$ has a non-inner automorphisms which we shell now
describe.

Let $\hat{T}\subseteq \Aut(\L_q)$ be the group consisting of
automorphisms which have the elements of the Chevalley basis as
eigenvectors and act as the identity on $\L_q^0$. The group
$\hat{T}$ is abelian, has exponent $q-1$ and contains $T$. In
addition, $\hat{T}$ normalizes $G_q$, so conjugation by its elements
induces automorphisms of $G_q$. The automorphisms which are induced
in this way are called diagonal automorphisms.

The second set of automorphisms are the graph automorphisms. These
automorphisms arise from symmetry of the Dynkin diagram of $\Pi$.
Since $p$ is a prime greater than $3$, the only Chevalley groups
which have a non-identity graph automorphism are of type $A_l$ for
$l \ge 2$, $D_l$ for $l \ge 4$ and $E_6$. We consider each case
separately.

Type $A_l$ with $l \ge 2$: In this case there is only one
non-identity symmetry of the graph. The symmetry sends $r_i$ to
$r_{l-i}$ for $1 \le i \le l$. The graph automorphism $\alpha_l$
associated with this symmetry has order $2$, it preserves $T$ and it
satisfies $\alpha_l(d_{r_i,t})=d_{r_{l-i},t}$ for $1 \le i \le l$
and $t \in \F_q^*$.

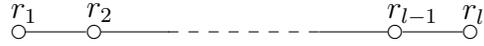
\begin{figure}[h!]
\begin{center}
\setlength{\unitlength}{3947sp}%
\begingroup\makeatletter\ifx\SetFigFont\undefined%
\gdef\SetFigFont#1#2#3#4#5{%
  \reset@font\fontsize{#1}{#2pt}%
  \fontfamily{#3}\fontseries{#4}\fontshape{#5}%
  \selectfont}%
\fi\endgroup%
\begin{picture}(2951,349)(881,80)
{\color[rgb]{0,0,0}\thinlines \put(946,130){\circle{84}}
}%
{\color[rgb]{0,0,0}\put(1419,129){\circle{84}}
}%
{\color[rgb]{0,0,0}\put(3309,129){\circle{84}}
}%
{\color[rgb]{0,0,0}\put(3782,129){\circle{84}}
}%
{\color[rgb]{0,0,0}\put(984,131){\line( 1, 0){385}}
}%
{\color[rgb]{0,0,0}\put(1462,131){\line( 1, 0){430}}
}%
{\color[rgb]{0,0,0}\multiput(1886,131)(126.00000,0.00000){8}{\line(
1, 0){ 63.000}}
}%
{\color[rgb]{0,0,0}\put(2834,131){\line( 1, 0){430}}
}%
{\color[rgb]{0,0,0}\put(3357,131){\line( 1, 0){377}}
}%
\put(896,220){\makebox(0,0)[lb]{\smash{{\SetFigFont{12}{14.4}{\rmdefault}{\mddefault}{\updefault}{\color[rgb]{0,0,0}$r_1$}%
}}}}
\put(1371,224){\makebox(0,0)[lb]{\smash{{\SetFigFont{12}{14.4}{\rmdefault}{\mddefault}{\updefault}{\color[rgb]{0,0,0}$r_2$}%
}}}}
\put(3260,217){\makebox(0,0)[lb]{\smash{{\SetFigFont{12}{14.4}{\rmdefault}{\mddefault}{\updefault}{\color[rgb]{0,0,0}$r_{l-1}$}%
}}}}
\put(3733,213){\makebox(0,0)[lb]{\smash{{\SetFigFont{12}{14.4}{\rmdefault}{\mddefault}{\updefault}{\color[rgb]{0,0,0}$r_l$}%
}}}}
\end{picture}%
\caption{Dynkin diagram of type $A_l$.}
\end{center}
\end{figure}

Type $D_4$: For every $\sigma\in\textrm{Sym}\{1,2,3,4\}$ with
$\sigma(1)=1$, there is a graph symmetry which transform $r_i$ to
$r_{\sigma(i)}$. The graph automorphism $\delta_\sigma$ associated
with this symmetry has the same order as $\sigma$, it preserves $T$
and it satisfies $\delta_\sigma(d_{r_i,t})=d_{r_{\sigma(i),t}}$ for
$1 \le i \le 4$ and $t \in \F_q^*$.

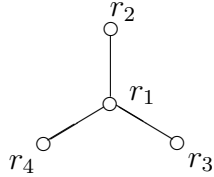
\begin{figure}[h!]
\begin{center}
%
%
\setlength{\unitlength}{3947sp}%
\begingroup\makeatletter\ifx\SetFigFont\undefined%
\gdef\SetFigFont#1#2#3#4#5{%
  \reset@font\fontsize{#1}{#2pt}%
  \fontfamily{#3}\fontseries{#4}\fontshape{#5}%
  \selectfont}%
\fi\endgroup%
\begin{picture}(1154,1202)(4313,-3182)
{\color[rgb]{0,0,0}\thinlines
\put(4963,-2705){\circle{84}}
}%
{\color[rgb]{0,0,0}\put(4968,-2231){\circle{84}}
}%
{\color[rgb]{0,0,0}\put(5386,-2946){\circle{84}}
}%
{\color[rgb]{0,0,0}\put(4546,-2955){\circle{84}}
}%
{\color[rgb]{0,0,0}\put(4966,-2657){\line( 0, 1){383}}
}%
{\color[rgb]{0,0,0}\put(5356,-2924){\line(-5, 3){351.471}}
}%
{\color[rgb]{0,0,0}\put(4924,-2721){\line(-5,-3){351.471}}
}%
\put(5083,-2675){\makebox(0,0)[lb]{\smash{{\SetFigFont{12}{14.4}{\rmdefault}{\mddefault}{\updefault}{\color[rgb]{0,0,0}$r_1$}%
}}}}
\put(4962,-2151){\makebox(0,0)[lb]{\smash{{\SetFigFont{12}{14.4}{\rmdefault}{\mddefault}{\updefault}{\color[rgb]{0,0,0}$r_2$}%
}}}}
\put(5452,-3097){\makebox(0,0)[lb]{\smash{{\SetFigFont{12}{14.4}{\rmdefault}{\mddefault}{\updefault}{\color[rgb]{0,0,0}$r_3$}%
}}}}
\put(4328,-3109){\makebox(0,0)[lb]{\smash{{\SetFigFont{12}{14.4}{\rmdefault}{\mddefault}{\updefault}{\color[rgb]{0,0,0}$r_4$}%
}}}}
\end{picture}%
\caption{Dynkin diagram of type $D_4$.}
\end{center}
\end{figure}

Type $D_l$ with $l \ge 5$: In this case there is only one
non-identity symmetry of the graph. The symmetry switches between
$r_{l-1}$ and $r_{l}$ and leaves the other fundamental roots fixed.
The graph automorphism $\delta_l$ associated with this symmetry has
order $2$, it preserves $T$  and it satisfies
$\delta_l(d_{r_i,t})=d_{r_i,t}$, $\delta_l(d_{r_{l-1},t})=d_{r_l,t}$
and $\delta_l(d_{r_{l},t})=d_{r_{l-1},t}$ for $1 \le i \le l-2$ and
$t \in \F_q^*$.

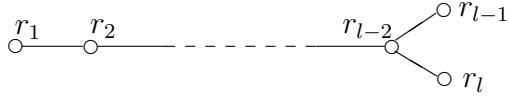
\begin{figure}[h!]
\begin{center}
%
%
\setlength{\unitlength}{3947sp}%
\begingroup\makeatletter\ifx\SetFigFont\undefined%
\gdef\SetFigFont#1#2#3#4#5{%
  \reset@font\fontsize{#1}{#2pt}%
  \fontfamily{#3}\fontseries{#4}\fontshape{#5}%
  \selectfont}%
\fi\endgroup%
\begin{picture}(2873,687)(896,-197)
{\color[rgb]{0,0,0}\thinlines
\put(946,130){\circle{84}}
}%
{\color[rgb]{0,0,0}\put(1419,129){\circle{84}}
}%
{\color[rgb]{0,0,0}\put(3309,129){\circle{84}}
}%
{\color[rgb]{0,0,0}\put(3637,359){\circle{84}}
}%
{\color[rgb]{0,0,0}\put(3640,-72){\circle{84}}
}%
{\color[rgb]{0,0,0}\put(984,131){\line( 1, 0){385}}
}%
{\color[rgb]{0,0,0}\put(1462,131){\line( 1, 0){430}}
}%
{\color[rgb]{0,0,0}\multiput(1886,131)(126.00000,0.00000){8}{\line( 1, 0){ 63.000}}
}%
{\color[rgb]{0,0,0}\put(2834,131){\line( 1, 0){430}}
}%
{\color[rgb]{0,0,0}\put(3334,171){\line( 3, 2){252.923}}
}%
{\color[rgb]{0,0,0}\put(3331, 99){\line( 5,-3){261.029}}
}%
\put(2999,216){\makebox(0,0)[lb]{\smash{{\SetFigFont{12}{14.4}{\rmdefault}{\mddefault}{\updefault}{\color[rgb]{0,0,0}$r_{l-2}$}%
}}}}
\put(3729,319){\makebox(0,0)[lb]{\smash{{\SetFigFont{12}{14.4}{\rmdefault}{\mddefault}{\updefault}{\color[rgb]{0,0,0}$r_{l-1}$}%
}}}}
\put(3754,-124){\makebox(0,0)[lb]{\smash{{\SetFigFont{12}{14.4}{\rmdefault}{\mddefault}{\updefault}{\color[rgb]{0,0,0}$r_l$}%
}}}}
\put(941,209){\makebox(0,0)[lb]{\smash{{\SetFigFont{12}{14.4}{\rmdefault}{\mddefault}{\updefault}{\color[rgb]{0,0,0}$r_1$}%
}}}}
\put(1414,214){\makebox(0,0)[lb]{\smash{{\SetFigFont{12}{14.4}{\rmdefault}{\mddefault}{\updefault}{\color[rgb]{0,0,0}$r_2$}%
}}}}
\end{picture}%
\caption{Dynkin diagram of type $D_l$.}
\end{center}
\end{figure}

Type $E_6$: Also in this case case there is only one non-identity
symmetry of the graph. The symmetry sends $r_i$ to $r_{5-i}$ for $1
\le i \le 5$ and leaves $r_6$ fixed. The graph automorphism
$\varepsilon$ associated with this symmetry has order $2$, it
preserves $T$ and it satisfies $\varepsilon(d_{r_1,t})=d_{r_5,t}$,
$\varepsilon(d_{r_2,t})=d_{r_4,t}$,
$\varepsilon(d_{r_3,t})=d_{r_3,t}$ and
$\varepsilon(d_{r_6,t})=d_{r_6,t}$ for every $t \in \F_q^*$.

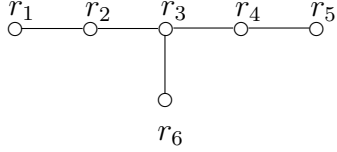
\begin{figure}[h!]
\begin{center}
%
%
\setlength{\unitlength}{3947sp}%
\begingroup\makeatletter\ifx\SetFigFont\undefined%
\gdef\SetFigFont#1#2#3#4#5{%
  \reset@font\fontsize{#1}{#2pt}%
  \fontfamily{#3}\fontseries{#4}\fontshape{#5}%
  \selectfont}%
\fi\endgroup%
\begin{picture}(1998,1029)(889,-644)
{\color[rgb]{0,0,0}\thinlines
\put(946,130){\circle{84}}
}%
{\color[rgb]{0,0,0}\put(1419,129){\circle{84}}
}%
{\color[rgb]{0,0,0}\put(1892,129){\circle{84}}
}%
{\color[rgb]{0,0,0}\put(2364,129){\circle{84}}
}%
{\color[rgb]{0,0,0}\put(2837,129){\circle{84}}
}%
{\color[rgb]{0,0,0}\put(1887,-318){\circle{84}}
}%
{\color[rgb]{0,0,0}\put(984,131){\line( 1, 0){385}}
}%
{\color[rgb]{0,0,0}\put(1464,131){\line( 1, 0){378}}
}%
{\color[rgb]{0,0,0}\put(1939,134){\line( 1, 0){375}}
}%
{\color[rgb]{0,0,0}\put(2411,134){\line( 1, 0){383}}
}%
{\color[rgb]{0,0,0}\put(1886, 91){\line( 0,-1){370}}
}%
\put(904,214){\makebox(0,0)[lb]{\smash{{\SetFigFont{12}{14.4}{\rmdefault}{\mddefault}{\updefault}{\color[rgb]{0,0,0}$r_1$}%
}}}}
\put(1384,211){\makebox(0,0)[lb]{\smash{{\SetFigFont{12}{14.4}{\rmdefault}{\mddefault}{\updefault}{\color[rgb]{0,0,0}$r_2$}%
}}}}
\put(1859,214){\makebox(0,0)[lb]{\smash{{\SetFigFont{12}{14.4}{\rmdefault}{\mddefault}{\updefault}{\color[rgb]{0,0,0}$r_3$}%
}}}}
\put(2326,209){\makebox(0,0)[lb]{\smash{{\SetFigFont{12}{14.4}{\rmdefault}{\mddefault}{\updefault}{\color[rgb]{0,0,0}$r_4$}%
}}}}
\put(2799,211){\makebox(0,0)[lb]{\smash{{\SetFigFont{12}{14.4}{\rmdefault}{\mddefault}{\updefault}{\color[rgb]{0,0,0}$r_5$}%
}}}}
\put(1839,-571){\makebox(0,0)[lb]{\smash{{\SetFigFont{12}{14.4}{\rmdefault}{\mddefault}{\updefault}{\color[rgb]{0,0,0}$r_6$}%
}}}}
\end{picture}%
\caption{Dynkin diagram of type $E_6$.}
\end{center}
\end{figure}

The last set of automorphisms are the field automorphisms. Every
$\varphi\in\aut(\F_q)$ induces an automorphism
$\bar{\varphi}\in\aut(G_q)$, which is called a field automorphism,
such that $\bar{\varphi}(x_{r,t}):=x_{r,\varphi(t)}$ and
$\bar{\varphi}(d_{r,t})=d_{r,\varphi(t)}$ for every $r\in\Phi$ and
$t\in \F_q$.

\begin{lemma}\label{stable group} Let $r_1,\cdots,r_l$ be the fundamental roots of
$\Phi$.  Denote $d_t:=d_{r_1,t}\cdots d_{r_l,t}$ for every $t \in
\F_q^*$.  The set $D:=\{d_t \mid t \in F_p^*\}$ is a subgroup of
$G_q$ which is pointwise fixed by every diagonal, graph or field
automorphism. Furthermore,  the size of $D$ is $p-1$.
\end{lemma}
\begin{proof} All the assertion follows directly from the definitions expect
the one about the size of $D$. If $s \in \Phi$ then
$d_t(e_s)=t^{m_s}e_s$ where $m_{s}:=(h_{r_1}+\cdots+h_{r_{l}},s)$.
Inspection of the Dynkin diagrams of all root systems of rank $l$
shows that there is always a fundamental root $s$ such that
$m_s:=(h_{r_1}+\cdots+h_{r_{l}},s)=-1$ so $d_t(s)=t^{-1}$. This
implies the $|D|=p-1$.
\end{proof}
\begin{lemma}\label{enough semi simple elements 2} Fix $n \in
\mathbb{N}^+$ and $\alpha \in (0,1)$ and define $D$ as in the above
lemma. There exists a constant $c$ such that if $p \ge c$ and $H$ is
a subgroup of $G_q$ such that $D \subseteq H \subseteq T$ then
$|H_n| \ge \alpha|H|$ where $H_n:=\{h \in H \mid
\cen_{G_q}(h^n)=T\}$.
\end{lemma}
\begin{proof}  Let $r_1,\cdots,r_l$
be the fundamental roots of $\Phi$ and $d_t$ as in the above lemma.
It suffices to show that if $p$ is large enough that for every $h
\in H$ the number of $t\in (\F_p^*)$ such that $hd_t\in H_n$ is at
least $\alpha(p-1)$. Fix $h \in H$ and assume $h(e_s)=\lambda_se_s$
for every $s\in \Phi$. If $s \in \Phi$ then
$$hd_t(e_s)=\lambda_st^{m_s}e_s$$
where $m_{s}:=(h_{r_1}+\cdots+h_{r_{l}},s)$. Denote $m:=\max
|m_{s}|$ where the maximum is taken over $s \in \Phi$. Lemma
\ref{reg semi simple} shows then if $t \in \F_p^*$ then $hd_t\in
H_n$ whenever the following two conditions hold:
\begin{itemize}
\item[1.] For every $s \in \Phi$ the product $\lambda_s^n t^{nm_{s}}$ is different from 1.
\item[2.] If $s_1\in \Phi^+$ and $s_2 \in \Phi^-$ then
$\lambda_{s_1}^n t^{nm_{s_1}}\ne \lambda_{s_2}^n t^{nm_{s_2}}$.
\end{itemize}
If $s\in \Phi^+$ ($s\in \Phi^-$) then $s$ is a linear sum of
elements of $\Pi$ where all the coefficients in this sum are
non-negative (non-positive). In addition, the scalar product of any
two distinct fundamental roots is non-positive and $s$ can not be
orthogonal to all the fundamental roots. Hence, if $s \in \Phi^+$
then $m_s<0$ while if $s \in \Phi^-$ then $m_s>0$. Thus, for $s \in
\Phi$ the number of $t \in \F_p^*$ which does not satisfy condition
1 is at most $nm$. Similarly, for $s_1\in \Phi^+$ and $s_2 \in
\Phi^-$ the number of $t \in \F_p^*$ which does not satisfy
condition 2 is at most  $2nm$. Therefore, the number of $t\in
\F_p^*$ with $hd_t\in H_n$ is at least
$$(p-1)-nm|\Phi|-\frac{nm}{2}|\Phi|^2$$ which is
 greater than $\alpha(p-1)$
for large enough $p$.
\end{proof}
The following is a first step to verify condition $2$ of Proposition
\ref{power proportion prop}.
\begin{corol}\label{enough fixed points2} Let $n\in \N^+$ and $\alpha \in (0,1)$.
There exists a constant $w$ such that for every $p\ge w$ and every
$\mu \in \aut(G_q)$ the following holds:

If $\mu$ stabilizes $T$ and pointwise fixes $D$ then $|T_{\mu,n}|\ge
\alpha|T_\mu|$. In particular, this it true for $\mu$ which is a
product of field, graph and diagonal automorphisms
\end{corol}

We close this subsection with a further discussion about
automorphisms.  As before let $p$ be a prime number greater than $3$
and let $q:=p^k$ for $k\in \N^+$. The group $G_q$ is simple and we
regard it as a normal subgroup of its automorphism group. Every
automorphism of $G_q$ is a product of the form
$\varphi\eta\kappa\iota$ where $\iota$ is an inner automorphism,
$\kappa$ is a diagonal automorphism, $\eta$ a graph automorphism and
$\varphi$ is a field automorphism. In Subsection \ref{subsection
-products} we will need a bound on the order of $\varphi\eta\kappa$
so we investigate this product. Graph and field automorphisms
commute and both normalize the group of diagonal automorphisms
$\hat{T}_q$. Hence, the order of $\varphi\eta\kappa$ is bounded by
the product $\ord(\varphi)\ord(\eta)\ord(\kappa)$ since the group of
diagonal automorphisms $\hat{T}$ is abelian. Thus, it is enough to
bound the order of $\varphi$, $\eta$ and $\kappa$ separately. The
order of a graph automorphism is at most $3$ while the order of a
field automorphism of $G_{q}$ divides $k$ (since $q=p^k$). However,
a diagonal automorphisms can have arbitrary large orders if $p$ is
large. To overcome this problem we note that $T \subseteq \hat{T}$
so in the presentation of an automorphism as a product
$\varphi\eta\kappa\iota$ the diagonal automorphism $\kappa$ can be
replaced by others automorphisms of the coset $\kappa T$. We need
the following group theoretic lemma:
\begin{lemma}\label{group theory} Let $G$ be a finite group which contains
a normal subgroup $H$. Then every coset of $H$ is $G$ has a
representative $g$ such that every prime divisor of $\ord(g)$ also
divides $[G:H]$.
\end{lemma}
\begin{proof} If $p$ is a prime which does not divide $[G:H]$
and $q$ is a power of $p$ then the set $R:=\{g^q \mid g \in G\}$
contains a representative of each coset. If in addition $q$ is large
enough then  $p$ does not divide the order of any one of the
elements of $R$.
\end{proof}

The index $[\hat{T}:T]$ depends on the type of the root system and
the field $\F_q$ and it is as follows:
$$\begin{array}{ccccccccc}
A_l & B_l & C_l & D_l & G_2 & F_4 & E_6 & E_7 & E_8 \\
\gcd (l+1,q-1) & 2 & 2 & (4,q^l-1) & 1 & 1 & \gcd (3,q-1) & 2 & 1
\end{array}$$

\begin{lemma}\label{bounded order representitive.} Let $k \in \N^+$ such that
$q:=p^k$. If $\xi\in\aut(G_q)$ then there are:  a field automorphism
$\varphi$ , a graph automorphism $\eta$, a diagonal automorphism
$\kappa$ and an inner automorphism $\iota$ such that
$\xi:=\varphi\eta\kappa\iota$ and the order of
$\mu:=\varphi\eta\kappa$ divides $6kz$ where $z$ divides $q-1$ and
every prime factor of $z$ divides $6(l+1)$.
\end{lemma}
\begin{proof} The above discussion implies that it is enough
to show that $\xi$ can be written in the form
$\xi=\eta\varphi\kappa\iota$ such that $z:=\ord(\kappa)$ satisfies
the required properties. The order of the elements of $\hat{T}$
divides $q-1$ and Lemma \ref{group theory} shows that $\kappa$ can
be chosen such that the prime factors of its order divide
$[\hat{T}:T]$. The above table shows that the primes which divide
$[\hat{T}:T]$ also divide $6(l+1)$.
\end{proof}

The point of the Lemma \ref{bounded order representitive.} is that
for a given $r \in \N^+$ there is a suitable arithmetic progression
$(a+bj)_{j\in \N}$ with coprime $a$ and $b$ such that if $q=p^k$
where $k \le r$ and $p$ is a prime which belongs to the arithmetic
progression then the product $\varphi\eta\kappa$ has bounded order.

\subsection{Steinberg groups}

Assume that $\Phi$ has a non trivial symmetry of the Dynkin diagram,
i.e, $\Phi$ is of one of the type: $A_l$ for $l\ge 2$, $D_l \ge 4$
or $E_6$. Let $\alpha\in \aut (G_q)$ be the graph automorphism
associated to such a symmetry, then $\ord(\alpha)=2$ unless $\Phi$
is of type $D_4$ where it is also possible that $\ord(\alpha)=3$.
Assume $q=p^k$ where $k \in \N^+$ is divisible by $\ord(\alpha)$ and
let $\beta$ be a field automorphism with the same order as $\alpha$
(recall that the automorphism group of $\F_{p^r}$ is cyclic of order
$r$). Let $U_\gamma$($V_\gamma$) be the subgroup of the unipotent
group $U$ ($V$) of the elements fixed by $\gamma:=\alpha\beta$. The
Steinberg group $G_q^*$ of type $\Phi$ over $\F_q$ is the subgroup
of $G_q$ generated by $U_\gamma$ and $V_\gamma$. The group $G_q^*$
is fixed by $\gamma$ but it can be properly contained in the
subgroup of fixed points. Note that for $p \ge 5$ the  Steinberg
groups are only defined when the root system is of type $A_l$ for
$l\ge 2$, $D_l \ge 4$ or $E_6$ and there is a field automorphism of
$\F_q$ of the same order as a non-trivial graph automorphism. If
$\Phi$ is of type different than $D_4$, then there is just one
non-trivial graph automorphism and its order is $2$, so $G_q^*$ is
uniquely defined and $q$ is a square. On the other hand, if the type
is $D_4$ then there are $3$ non-trivial graph automorphisms of order
$2$ and $2$ graph automorphisms of order $3$. However, it is easily
verified that up to isomorphism the Steinberg group depends only on
the order of the graph automorphism. Thus, if $6$ divide $k$ and
$q=p^k$ the symbol $G_q^*$ can represent two different Steinberg
groups. This will not cause us problems since both groups share the
properties we are concerned with.

A common notation for the Steinberg groups is ${}^2A_l(q^2)$,
${}^2D_l(q^2)$, ${}^3D_4(q^3)$ and ${}^2E_6(q^2)$ where for instance
${}^3D_4(q^3)$ is the Steinberg group corresponding to the graph
automorphism of order $3$ of $G_{q^3}$. We prefer not to use this
notation since our arguments does not depend on the type of the root
system. On the other hand, our arguments will depend on the fact
that the Steinberg groups are subgroups of the Chevalley groups and
the symbol $G_q^*$ emphasis this.

Let $G_q^*$ be an  Steinberg subgroup. Define $T^*:=T \cap G_q^*$
and $N^*:=N \cap G_q^*$ where $T$ and $N$ are the subgroup of $G_q$
defined above. The  proof of theorem $13.7.2$ in \cite{Ca} shows
that $D$ is a subgroup of $T^*$ where $D$ is defined in Lemma
\ref{stable group}. If $p$ is large enough then Corollary
\ref{enough fixed points2} shows that $D$ contains an element whose
centralizer is $T$, so $T^*$ is a maximal abelian subgroup of
$G_q^*$. In turn, Lemma \ref{mormalizer lemma} implies that
$N^*=N_{G_p^*}(T^*)$. In particular, we have $[N_{G_p^*}(T^*):T^*]
\le [N:T]=|W(\Phi)|$ where $W(\Phi)$ is the Weyl group of $\Phi$.

Next, we want to discuses the automorphisms of the Steinberg group
$G_q^*$ (see \cite{St}). The group $G_q^*$ has a trivial center
(since $p \ge 5$) so we can view it as a subgroup of its
automorphism group. The group of diagonal automorphisms $\hat{T}^*$
of $G_q^*$ consists of the restrictions of the automorphisms which
belong to $\hat{T}$ and stabilize $G_q^*$. Note that diagonal
automorphisms of $G_q^*$ fix $T^*$ pointwise. The field
automorphisms of $G_q^*$ are the restrictions of the field
automorphism of $G_q$, they also stabilize $T^*$. Lemma \ref{enough
semi simple elements 2} gives an analog of Corollary \ref{enough
fixed points2}:

\begin{corol}\label{enough fixed points3} Let $n\in \N^+$ and $\alpha \in (0,1)$.
There exists a constant $c$ such that for every $p\ge c$ and every
$\mu^* \in \aut(G_q^*)$ the following holds:

If $\mu^* $ stabilizes $T^*$ and pointwise fixes $D$ then
$|T^*_{\mu^*,n}|\ge \alpha|T^*_{\mu^*}|$. In particular, this is
true for $\mu$ which is a product of field and diagonal
automorphisms.
\end{corol}

As for Chevalley groups, every automorphism of $G_q^*$ is a product
of the form $\varphi\kappa\iota$ where $\iota$ is an inner
automorphism, $\kappa$ is a diagonal automorphism and $\varphi$ is a
field automorphism (there in no need for graph automorphisms). Note
that this fact does not follow directly from the equivalent fact for
Chevalley groups and requires a sperate proof which can be found in
\cite{St}. The index of $T^*$ in $\hat{T}^*$ is given in the
following table:
$$\begin{array}{cccc}
{}^2A_l(q^2) & {}^2D_l(q^2) & {}^3D_4(q^3) & {}^2E_6(q^2)  \\
\gcd (l+1,q+1) & \gcd (4,q^l+1)& 1 & (3,q+1)
\end{array}$$

We get an analog of lemma \ref{bounded order representitive.}:

\begin{lemma}\label{bounded order representitive2} Let $k \in \N^+$ such that
$q=p^k$. If $\xi\in\aut(G_q^*)$, then there are: a field
automorphism $\varphi$ , a diagonal automorphism $\kappa$ and an
inner automorphism $\iota$ such that $\xi:=\varphi\kappa\iota$ and
the order of $\mu^*:=\varphi\kappa$ divides $kz$ where $z$ divides
$q-1$ and every prime factor of $z$ divides $6(l+1)$.
\end{lemma}

With the notations as above, we conclude:

\begin{corol}\label{final reduction} Let $d\in \N^+$ and fix a prime number $m$. Assume that $p$ is a large enough prime which
belongs to the series $(a+bj)_{j \in \N}$ and $q = p^k$ for $k \le
d$ where $a:=1+6m(l+1)^2d!^2$ and $b:=36m^2(l+1)^4d!^4$. Then for
every $\xi \in \aut(G_q)$ and every $\xi^* \in \aut(G_q^*)$ there
are $\iota\in\inn(G_q)$, $\iota^*\in\inn(G_q^*)$, $\mu \in
\aut(G_q)$ and $\mu^* \in \aut(G_q^*)$ such that:
\begin{itemize}
\item[1.] $\xi=\mu\iota$ and $\xi^*=\mu^*\iota^*$.
\item[2.] $\mu(T)=T$ and $\mu^*(T^*)=T^*$.
\item[3.] both $\mu$ and $\mu^*$ pointwise fix $D$.
\item[4.] $m$ divides the order of $D$.
\item[5.] $\ord(\mu)$ and $\ord(\mu^*)$ divides $216(l+1)^3d!^4$.
\end{itemize}
\end{corol}
\begin{proof} Fix $\xi$ and $\xi^*$ and assume that $q=p^k$ for $k \le d$.
Lemmas \ref{bounded order representitive.} and \ref{bounded order
representitive2} show that we can find $\iota\in\inn(G_q)$,
$\iota^*\in\inn(G_q^*)$, $\mu \in \aut(G_q)$ and $\mu^* \in
\aut(G_q^*)$ such that conditions 1,2,3 hold and the order of $\mu$
and $\mu^*$ divide $6zk$ where $z$ divides $q-1$ and every prime
factor of $z$ divides $6(l+1)$. The above arithmetic progression
implies that
$$q-1=p^k- 1\equiv 6mk(l+1)^2d!^2 \ \  (\Mod\ 36m^2(l+1)^4d!^4 ).$$ Thus, $z$
divides $6mk(l+1)^2d!^2$ so it also divides $36(l+1)^3d!^3$. It
follows that $\ord(\mu)$ and $\ord(\mu^*)$ divides $216(l+1)^3d!^4$.
Finally, the size of $D$ is $p-1$ which is divisible by $m$.
\end{proof}

\subsection{Powers in extension of finite simple Lie
groups}\label{subsection -products} The main goal of this section is
to prove Theorem \ref{power main4} below. We start this section with
a general discussion about powers in extension of finite groups. Let
$H$ be a finite group. Let $G$ be a normal subgroup of $H$ with
trivial center. Let $K$ be a coset of $G$ in $H$. Our goal is to
bound the size of $\{k^m \mid k \in K\}$ for some $m \in \N^+$. Fix
some $k \in K$ and let $\zeta \in \aut(G)$ be the automorphism
induced by conjugation by $k$. Then
$$|\{k^m \mid k \in K\}|=|\{ (g\zeta)^m \mid g \in G \}|$$ where in the right
side $G$ is viewed as a subgroup of $\aut(G)$ . Hence, we only have
deal with groups of the later form, i.e the case where $H=\aut(G)$.

Next, we focus on the case $G=S^r$ where $S$ is a non-abelian finite
simple group and $r \in \N^+$. Every automorphism $\zeta$ of the
direct product $S^r$ is of the form:
\begin{equation}\label{simple product}\zeta(s_1,\cdots,s_r) =
(\xi_1(s_{\sigma(1)}),\cdots,\xi_r(s_{\sigma(r)}))\end{equation}
where $\xi_1,\cdots,\xi_r \in \aut(S)$ and $\sigma\in \sym(r)$.
Every permutation is a product of disjoint cycles ,say, $\sigma$ is
a product of $k$ cycles of lengthes $r_1,\cdots,r_k$. By renumbering
the copies of $S^r$ and using the isomorphism $S^r \simeq
S^{r_1}\times\cdots \times S^{r_k}$ we can assume that there are
$\zeta_i \in \aut(S^{r_i})$ and $\xi_{i,j}\in \aut(S)$ for $1 \le i
\le k$ and $1 \le j \le r_i$ such that
$$\zeta(\bar{s_1},\cdots,\bar{s_k}) =
(\zeta_1(\bar{s}_1),\cdots,\zeta_k(\bar{s}_k))$$ and
$$\zeta_i(s_{i,1},\cdots,s_{i,r_i})=(\xi_{i,1}(s_{i,\sigma_i(1)}),\cdots,\xi_{i,r_i}(s_{i,\sigma_i(r_i)}))$$
where $\bar{s}_i=(s_{i,1},\cdots,s_{i,r_i}) \in S^{r_i}$ and
$\sigma_i=(r_i \ r_{i-1}\cdots 2\ 1 )\in \sym(r_i)$. Note that if
$$|\{(\bar{s}_1 \zeta_1)^m \mid \bar{s}_1 \in  S^{r_1} \}| \le
c|S|^{r_1}$$ for some constant  $c>0$ then also $$|\{(\bar{s}
\zeta)^m \mid \bar{s} \in S^{r} \}| \le c|S|^{r}.$$ This allows us
to restrict to the case where $\sigma$ is a cyclic permutation.

Finally, let $\zeta \in \aut(S^r)$ as in equation \ref{simple
product} where  $\sigma$ is the permutation  $(r \cdots 1)$. Denote
$\chi:=(\rho_1,\rho_2,\cdots,\rho_r) \in \aut(S^r)$ where
$\rho_j:=\xi_j\cdots\xi_2$ for $j \ge 2$ and $\rho_1:=\id$.
Conjugation by $\chi$ allows us to replace $\zeta$ with
$\chi^{-1}\zeta\chi$, i.e. to assume that we have
$\xi_2=\cdots=\xi_r=\id$.

The main theorem of this section is:
\begin{thm}\label{power main4} Let $d,l.r\in \N^+$ be constants. There
is a constant $c \in \N^+$ such that for every number $m \ge 2$ the
following claim holds:

Let $p$ be a prime which belongs to the arithmetic progression
$(a+bj)_{j \ge c}$ where
$$a:=1+6m(l+1)^2d!^2 \text{ and } b:=36m^2(l+1)^4d!^4.$$ Let $\Gamma$ be a finite group with a non-trivial
normal subgroup $\Lambda$. Furthermore, assume that $\Lambda$  is
isomorphic to a product of at most $r$ finite quasi simple groups of
Lie type (not necessarily of adjoint type) of rank at most $l$ over
extensions of $\F_{p}$ of degree at most $d$. Then for every coset
$\Psi$ of $\Lambda$ the following inequity holds:
$$|\{g^m \mid g \in \Psi\}| \le \left(1-\frac{1}{4n^r{C_l}^r}\right)|\Psi|$$
where $n:=216(l+1)^3d!^4$ and $C_l$ is the maximal size of a Weil
group of rank $l$.
\end{thm}
\begin{proof} We start with some reductions. The group
$\Lambda$ has a subgroup $N$ normal in $\Gamma$ such that
$\Lambda/N$ is isomorphic to a product of isomorphic finite simple
Lie groups of adjoint type. We replace $\Lambda$ with this quotient.
We focus on the case where the finite simple Lie group is an
Steinberg group $G_q^*$ of type $\Phi$ over $\F_q$ with $q:=p^d$ and
$\Lambda=(G_q^*)^r$. The proofs of the other cases are similar.
Furthermore, we can assume that $m$ is a prime number.

Define $n=216r(l+1)^3d!^4$. Corollary \ref{enough fixed points3}
allows us to choose a constant $c$ such that if $p$ belongs to the
above series then $|T_{\mu^*,n}|\ge
\frac{3}{4}^{\frac{1}{r}}|T_{\mu^*}|$ for every $\mu^*\in
\aut(G_q^*)$ which stabilizes $T$ and pointwise fixes $D$. We can
further assume that $c$ is large enough so that $T^*$ is a
maximal-abelian subgroup of $G_q^*$ and $N_{G_{q^*}}(T^*)=N^*$.

The discussion before the proof shows that it is enough to deal with
the case $\Gamma=\aut(\Lambda)$ and $\Psi=\Lambda \zeta^*$ where
$$\zeta^*(s_1,\cdots,s_r) =
(\xi^*(s_r),s_1,\cdots,s_{r-1})$$ for some $\xi^* \in \aut(G_q^*)$.
Corollary \ref{final reduction} allows us to replace $\zeta^*$ with
some representative of $\Lambda\zeta^*$ such that the new
representative, still denoted by $\zeta^*$, stratifies
$$\zeta^*(s_1,\cdots,s_r) =
(\mu^*(s_r),s_1,\cdots,s_{r-1})$$ where $\mu^* \in \aut(G_q^*)$
satisfies conditions $2$, $3$, $4$ and $5$ of that corollary. This
implies that the order of $\zeta^*$ divides $n$.

We are in position to verify that the requirements of Proposition
\ref{power proportion prop}. The group $\tilde{T}:=(T^*)^r$ is a
maximal-abelian subgroup of $\Lambda$ and its normalizer is
$\tilde{N}:=N_{\Lambda}(\tilde{T})=(N^*)^r$. In particular,
$[\tilde{N}:\tilde{T}]\le |W(\Phi)|^r$. Note that
$\tilde{T}_{\zeta^*,\ord(\zeta^*)}\supseteq \tilde{T}_{\zeta^*,n}$
since $\ord(\zeta^*)$ divides $n$. Thus,
\begin{itemize}
\item $|\tilde{T}_{\zeta^*,\ord(\zeta^*)}| \ge
\frac{3}{4}|\tilde{T}_{\zeta^*}|$ since
$\tilde{T}_{\zeta^*}=(T_{\mu^*})^r$ and $\tilde{T}_{\zeta^*,n}
=(T^*_{\mu^*,n})^r$.
\item $m$ divides the order of $\tilde{T}_{\zeta^*}$ since it contains the subgroup
$D^r$  and $m$ divides $|D|$.
\end{itemize}
Proposition \ref{power proportion prop} together with the last two
points shows that
$$|\{g^m \mid g \in \Psi\}| \le \left(1-\frac{1}{4c_n|W(\Phi)|^r}\right)|\Psi|$$
where $c_n$ is the number of elements of $\tilde{T}$ such that their
orders divide $n$. However, $\tilde{T}=(T^*)^r \subseteq T^r
\subseteq \hat{T}^r \simeq (\F_q^*)^r$ where $\F_q^*$ is the
multiplicative group of $\F_q$ so $c_n \le n^r$.
\end{proof}

The next Corollary is stated with the notations of Proposition
\ref{structre}.
\begin{corol}\label{done} For $n \ge 2$ every there is a constant $c$
for which the following claim holds:

Let $\Gamma$ be a subgroup of $\GL_n(\Q)$ such that its
Zariski-closure is semisimple. Let $k \in \N^+$ be large enough and
let $2 \le m \le k^2$. Then there are two coprime natural numbers $2
\le a,b \le k^5$ such that for every prime $p$ which belongs to the
arithmetic progression $(a+bj)_{j \ge 1}$ and every two coset $C,D
\in \Gamma_p/\Lambda_p$, the size of $\{g^m\mid g \in D \}\cap C$ is
at most $(1-c)|\Lambda_p|$.

\end{corol}

\end{document}